\documentclass{article}
\usepackage[english]{babel}
\usepackage[utf8]{inputenc}
\usepackage{bm}
\usepackage{amsmath,amsthm,amsfonts}
\usepackage{graphicx}
\usepackage{color}
\usepackage{multirow}
\usepackage{subcaption}

\newtheorem{theorem}{Theorem}

\newcommand{\const}{\mathop{\rm const}\nolimits}

\graphicspath{ {./figs/} }

\newcommand{\rev}[1]{{\color{black}#1}}

\date{}

\title{Implicit-Explicit schemes for decoupling multicontinuum problems in porous media}

\author{
Maria Vasilyeva 
\thanks{Department of Mathematics and Statistics, Texas A\&M University  - Corpus Christi,   Corpus Christi, Texas, USA. Email: {\tt maria.vasilyeva@tamucc.edu}.}
}

\begin{document}

\maketitle

\begin{abstract}
In this work, we present an efficient way to decouple the multicontinuum problems. To construct decoupled schemes, we propose Implicit-Explicit time approximation in general form and study them for the fine-scale and coarse-scale space approximations. 
We use a finite-volume method for fine-scale approximation, and the nonlocal multicontinuum (NLMC) method is used to construct an accurate and physically meaningful coarse-scale approximation. The NLMC method is an accurate technique to develop a physically meaningful coarse scale model based on defining the macroscale variables. The multiscale basis functions are constructed in local domains by solving constraint energy minimization problems and projecting the system to the coarse grid. The resulting basis functions have exponential decay properties and lead to the accurate approximation on a coarse grid. We construct a fully Implicit time approximation for semi-discrete systems arising after fine-scale and coarse-scale space approximations. 
We investigate the stability of the two and three-level schemes for fully Implicit and Implicit-Explicit time approximations schemes for multicontinuum problems in fractured porous media. We show that combining the decoupling technique with multiscale approximation leads to developing an accurate and efficient solver for multicontinuum problems. 
\end{abstract}

\section{Introduction}

Multicontinuum models occur in many real-world applications, such as geothermal energy production, unconventional oil and gas production, carbon capture and storage (CCS) in geological formations, deposition of nuclear waste, wastewater treatment, and many more \cite{ruiz2016mixed, hoteit2008efficient, karimi2001numerical, ctene2016algebraic}. 
In reservoir simulation, the multicontimuum models are used to describe complex interaction between multiple scales of heterogeneity such as hydraulic and natural fractures, vugs and cavities, and porous matrix \cite{barenblatt1960basic, warren1963behavior, kazemi1969pressure, pruess1985practical}. In \cite{arbogast1990derivation}, the general form of the dual-continuum model is derived from homogenization theory. 
Multicontinuum systems are characterized by high contrast properties between continua, and each continua may itself have a complex heterogeneous structure. Moreover, fractured porous media usually have complex fracture geometries and minimal thicknesses compared to typical reservoir sizes. A typical model approach for fractured media is based on the lower-dimensional representation of the fracture objects   \cite{martin2005modeling, d2012mixed, formaggia2014reduced, Quarteroni2008coupling, schwenck2015dimensionally}. In such mathematical models, we have a coupled mixed-dimensional multicontinuum system of equations.

Numerical simulation of flow in multicontinuum fracture porous media is a challenging task. Due to the multiple scales and complex geometries of the fractures, the simulation of the processes in fractured porous media requires fine grids for accurate approximation, which is computationally expensive. Moreover, the significant contrast in the continuum's permeability makes numerical simulations even more computationally expensive.
To construct efficient solvers, the multiscale methods are developed to reduce the dimension of the discrete system and perform simulations with fewer degrees of freedom than classical approximation methods such as finite volume, finite element, or finite difference methods \cite{houwu97, eh09, weinan2007heterogeneous, lunati2006multiscale, jenny2005adaptive}.
Several multiscale methods are developed to solve problems in fractured porous media. In \cite{hajibeygi2011hierarchical}, an iterative multiscale finite volume (i-MSFV) method is developed for flow in fractured porous media. The process is based on the Multiscale Finite Volume Method \cite{jenny2003multi, jenny2005adaptive, hajibeygi2008iterative}. 
A multiscale restriction smoothed basis method (F-MsRSB) for multiphase flow in heterogeneous fractured porous media is devised in \cite{shah2016multiscale}. 
In our previous works, we presented the multiscale method for flow in fractured porous media \cite{akkutlu2015multiscale, chung2017coupling, efendiev2015hierarchical, akkutlu2018multiscale}. The proposed approach is based on the Generalized Multiscale Finite Element Method (GMsFEM) and utilizes local spectral problems in multiscale basis construction \cite{efendiev2013generalized, CELV2015, chung2016adaptive}. The GMsFEM is effectively extended to solve problems in multicontinuum fractured porous media  \cite{vasilyeva2019multiscale, spiridonov2020generalized, tyrylgin2020generalized}. 
Recently, the Constraint Energy Minimization Generalized Multiscale Finite Element Method (CEM-GMsFEM) was proposed in \cite{chung2017constraint}. In CEM-GMsFEM, constructing multiscale basis functions starts with an auxiliary multiscale space that is defined by solving local spectral problems. Then, a constraint energy minimization is used to build multiscale basis functions in the oversampling domain. 
In CEM-GMsFEM, the choice of the auxiliary space can be made using careful design of the macroscale parameters. Moreover, the resulting construction can lead to a physically meaningful coarse scale model. In \cite{chung2017non, vasilyeva2019nonlocal}, we recently presented a nonlocal multicontinuum (NLMC) method for problems in fractured media. In the NLMC approximation, the resulting system is very similar to the traditional finite volume method but provides an accurate approximation due to the nonlocal coupling in and between each continuum. 
The NLMC method effectively solves different applied problems in multicontinuum media \cite{vasilyeva2019constrained, vasilyeva2019upscaling}.

To develop an efficient numerical implementation for multicontinuum problems, we construct decoupled (splitted) schemes to separate equations for each continuum. Decoupled discrete schemes can significantly reduce the computational work and allow using different methods and software for the subproblem solution.  
The main idea of the splitting technique is to separate the original problem into several smaller subproblems, which can be designed for efficient computational implementation. In time approximation, additive schemes are a helpful tool for solving unsteady equations \cite{ascher1995implicit, keyes2013multiphysics, vabishchevich2013additive, steefel2018approaches}.
The application of the additive schemes for multiscale approximation has been considered in \cite{efendiev2021temporal,efendiev2021splitting}. The partially-explicit time discretization for nonlinear multiscale problems is presented in \cite{chung2021contrast}. The method was combined with the machine learning techniques applied for the implicit part of the operator \cite{efendiev2022efficient}. 
The next extension of the partially explicit scheme is presented in \cite{leung2022multirate}, where the mutirate partially explicit scheme for multiscale flow problems is developed and analyzed. 
In \cite{vasilyeva2023efficient}, we proposed an efficient decoupled scheme for multicontinuum flow problems in fractured porous media. The presented approach is based on the additive representation of the operator with implicit-explicit approximation by time to decoupled equations for each continuum. We developed, analyzed, and investigated three first-order decoupled schemes for solving the classical multicontinuum problems in fractured porous media on fine grids with finite volume approximation by space. We extend the continuum decoupling approach for multiscale multicontinuum problems with nonlocal multicontinuum approximation on the coarse grid.

In this work, we develop decoupled multiscale schemes for solving multiscale multicontinuum problems in general form. With the help of an implicit-explicit time approximation, we decouple the fine-scale and coarse-scale discrete systems to construct an efficient and robust numerical algorithm. This work is motivated by \cite{vabishchevich2013additive, gaspar2014explicit} where implicit-explicit schemes are utilized to the splitting of time-dependent problems by an additive representation of the problem operator.  
The Implicit-Explicit (ImEx) schemes are widely used for diffusion-reaction or diffusion-convection problems. For the convection-diffusion problems, an explicit scheme is applied for the convection part, and an implicit scheme is utilized for the diffusion part because it gives a more substantial restriction on the time step size in the explicit approximation. In the convection-diffusion operator, the reason to approximate explicitly can be delivered to the non-symmetric part of the matrix or Jacobian for nonlinear systems for a general case.
In reaction-diffusion problems, the reaction part can be approximated explicitly \cite{vabishchevich2012explicit, vasilyeva2023uncoupling}.  
In \cite{southworth2023implicit} an  implicit-explicit scheme is constructed for radiation hydrodynamics, where coupled radiation equations are treated using implicit integration and hydrodynamics is treated by explicit integration.
Moreover, in such applications, ImEx schemes are used to decouple the slow and fast parts of the operator. The stiff part uses the implicit approximation to avoid small-time stepping. Furthermore, the implicit part is typically associated with a  diffusion operator and can be solved effectively by iterative methods. 
In \cite{ascher1995implicit}, linear multistep ImEx schemes are developed for general equations. Such schemes can be constructed for spatially discretized partial differential equations. In \cite{vabishchevich2013additive, samarskii2002additive}, the fundamental aspects of the splitting schemes are discussed for the problems described by time-dependent partial differential equations. The construction of the splitting schemes is based on the additive representation of the operator of the time-dependent problem, where various schemes are developed based on the implicit-explicit approximation. 
Due to the additive representation of an operator of a time-dependent problem as a sum of operators with a simpler structure, the transition to a new time level can be performed as a sequence of subproblems.

Motivated by Implicit-Explicit schemes in \cite{ascher1995implicit} and additive schemes for time-dependent problems in  \cite{vabishchevich2013additive, gaspar2014explicit}, we develop decoupling schemes in general form for multicontinuum problems with fine-scale and coarse-scale approximations by space. The operator is split into the block diagonal and off-block-diagonal parts to apply it to decouple the high-contrast continua.  
In this work, we continue the development of the efficient decoupling schemes proposed in our previous work \cite{vasilyeva2023efficient} and propose a general form of the Implicit-Explicit schemes. 
We note that the first-order schemes are similar to the scheme proposed in \cite{efendiev2021splitting, efendiev2021temporal}. The paper's main novelty is developing a general form of the decoupling scheme for fine-scale and coarse-scale problems. We systematically analyze the stability and computational performance of the proposed scheme for two and three-level schemes. Finally, numerical experiments demonstrate that the proposed implicit-explicit two and three-level schemes give an efficient decoupling approach for multicontinuum problems, reduce the number of iterations and size of the discrete problems, and lead to a highly accurate and computationally efficient algorithm.

The paper is organized as follows. Section 2 describes problem formulation for flow in multicontinuum media in a general form and presents a semi-discretization based on finite-volume method (FVM) and nonlocal multicontinua method (NLMC). The Implicit time approximation that leads to the coupled system of equations is given in Section 3. Section 4 presents the Implicit-Explicit schemes to decouple the multicontinuum system. Stability analysis is provided for two- and three-level Implicit and Implicit-Explicit schemes. Numerical results are given in Section 5. Finally, the conclusion is presented in Section 6.

\section{Problem formulation}
 
The multicontinuum models are widely used in reservoir simulations. For example,  in gas production from shale formation, we have a highly heterogeneous and complex mixture of organic matter, inorganic matter, and multiscale fractures \cite{akkutlu2012multiscale, akkutlu2017multiscale}. Another example is the fractured vuggy reservoirs, where multicontinuum models are used to characterize the complex interaction between vugs, fractures, and porous matrix \cite{wu2011multiple, wu2007triple, yao2010discrete}.
The mathematical model of flow in multicontinuum media is described by a coupled system of equations
\begin{equation}
\label{mm}
c_{\alpha} \frac{ \partial  u_{\alpha}}{\partial t}
- \nabla \cdot  (k_{\alpha}  \nabla u_{\alpha})
+  \sum_{\beta \neq \alpha} \sigma_{\alpha \beta} (u_{\alpha} - u_{\beta})  = f_{\alpha},  \quad 0 < t  \leq T, \quad x \in \Omega_{\alpha},
\end{equation}
where $\Omega_{\alpha}$ is a computational domain, $\alpha, \beta = 1,2,...,L$ and $L$ is the number of continuum. 
Here 
$u_{\alpha} = u_{\alpha}(x, t)$ is the pressure of ${\alpha}$ continuum, 
$c_{\alpha}$ and $k_{\alpha}$ are the problem coefficients, 
$\sigma_{\alpha \beta} = \sigma_{\beta \alpha}$  is the coupling coefficient between continuum $\alpha$ and $\beta$ that characterize a flow between them,  
$f_{\alpha}$ is the source terms.  

The system of equations \eqref{mm} is considered with some given initial conditions and homogeneous Neumann boundary conditions for each continuum
\begin{equation}
\label{mm-bc}
- k_{\alpha}  \nabla u_{\alpha} \cdot n = 0, \quad 0 < t  \leq T, \quad x \in \partial \Omega_{\alpha},
\end{equation}
where $n$ is the outer normal vector to the domain boundary $\partial \Omega_{\alpha}$.

Let $u \in V$ and $V = V_1 \oplus V_2 \oplus ... \oplus V_L$  be a direct sum of spaces $V_{\alpha}$, where $u_{\alpha} \in V_{\alpha}$ and  $V_{\alpha}$ is a Hilbert space.   
Then for $u(t) = (u_1, u_2, ...u_L)^T \in V$, we have the following system of equations
\begin{equation}
\label{mm1}
\mathcal{C} \frac{du}{dt} + \mathcal{A} u = f(t),   \quad 0 < t  \leq T,
\end{equation}
with 
\[
\mathcal{A} = 
\mathcal{D} + \mathcal{Q}, 
\quad
\mathcal{D} =
 \begin{pmatrix}
\mathcal{D}_1 & 0 & ... & 0\\
0 & \mathcal{D}_2 & ... & 0\\
... & ... & ... & ...\\
0 & 0 & ... & \mathcal{D}_L
\end{pmatrix}, 
\quad 
\mathcal{C} = 
\begin{pmatrix}
c_1 & 0 & ... & 0\\
0 & c_2 & ... & 0\\
... & ... & ... & ...\\
0 & 0 & ... & c_L
\end{pmatrix}, 
\]
and 
\[
\mathcal{Q} =  
 \begin{pmatrix}
\sum_{\beta \neq 1} \sigma _{1\beta} & -\sigma_{12} & ... & -\sigma_{1L}\\
-\sigma_{21} & \sum_{\beta \neq 2} \sigma_{2\beta} & ... & -\sigma_{2L}\\
... & ... & ... & ...\\
-\sigma_{L1} & -\sigma_{L2} & ... &  \sum_{\beta \neq L} \sigma_{L\beta}\\
\end{pmatrix}, 
\quad 
f = \begin{pmatrix}
f_1\\
f_2\\ 
...\\
f_L
\end{pmatrix}, 
\]
where $f_{\alpha} \in L^2(0,T;L^2(\Omega_{\alpha}))$ and $\mathcal{D}_{\alpha} u_{\alpha} = - \nabla \cdot (k_{\alpha} \nabla u_{\alpha})$ is the diffusion operator for component $\alpha$.  

Let $(u,v)$ and $(u,v)_{\mathcal{A}} = (\mathcal{A}u, v)$ be the scalar products for $u,v \in V$,  $||u|| = \sqrt{(u,u)}$ and $||u||_{\mathcal{A}} = \sqrt{ (\mathcal{A}u, v)}$ be the norms in $V$ with $(u,v) = \sum_{\alpha=1}^L (u_{\alpha}, v_{\alpha})$, $u_{\alpha} \in L^2(0,T;L^2(\Omega_{\alpha}))$. 
For the system of equations \eqref{mm1} that describe the flow problems in multicontinuum media,  we have the following physical properties
\begin{equation}
\label{mas}
c_{\alpha} \geq c_{0,\alpha} > 0, \quad 
k_{\alpha} \geq k_{0,\alpha} > 0, \quad 
\sigma_{\alpha \beta} = \sigma_{\beta \alpha}  \geq 0,
\end{equation}
with $c_{\alpha}, \sigma_{\alpha \beta}   \in L^{\infty}(\Omega_{\alpha})$ and $k_{\alpha}  \in W^{1, \infty}(\Omega_{\alpha})$. 
Therefore  $\mathcal{A}$ and $\mathcal{C}$  are self-adjoint and positive definite operators.

\subsection{Fine-scale approximation by space (FVM)}
 
In domain $\Omega$, we construct a structured grid $\mathcal{T}_h(\Omega)  = \cup_i \varsigma_i$, where $\varsigma_i$ is the square cell with length $h$ in each direction. 
For the problem \eqref{mm}, we use a finite volume approximation and obtain the following semi-discrete form
\begin{equation}
\label{app}
 c_{\alpha,i} \frac{\partial u^h_{{\alpha}, i} }{\partial t} |\varsigma^{\alpha}_i|
 + \sum_{j}  T_{{\alpha},ij}  (u^h_{{\alpha}, i} - u^h_{{\alpha}, j})
 + \sum_{{\alpha \neq \beta}} 
 \sum_{j} \sigma_{\alpha \beta,ij} (u^h_{{\alpha}, i} - u^h_{\beta, j} )
 =  f_{\alpha,i}   |\varsigma^{\alpha}_i|, \quad \forall i = 1, N^{\alpha}_h,
\end{equation}
where 
$T_{\alpha,ij} = k_{\alpha} |E_{ij}|/d_{ij}$ ($|E_{ij}|$ is the length of facet between cells $\varsigma_i$ and $\varsigma_j$, $d_{ij}$ is the distance between midpoint of cells $\varsigma_i$ and $\varsigma_j$),
$\sigma_{il} = \sigma |E^i_l|/d^i_l$ if $\iota_l \cap \varsigma_i \neq 0$ and zero else ($|E^i_l|$ is the length of continuum interface and $d^i_l$ is the distance between continuum), 
$N^{\alpha}_h$ is the number of cells and $\alpha = 1,2,...,L$. 

Let $u^h = (u^h_1, u^h_2, ...,u^h_L)$ and 
$u^h_{\alpha} = (u^h_{\alpha, 1}, u^h_{\alpha, 2},..., u^h_{\alpha, N^{\alpha}_H})^T$. Then we have the following system of coupled equations  in the matrix form
\begin{equation}
\label{app-mc}
M^h \frac{\partial u^h}{\partial t} + A^h u^h = F^h,
\end{equation}
with
\[
M^h = \begin{pmatrix}
M^h_1 & 0 & ... & 0 \\
0 & M^h_2 & ... &  0 \\
 ... &  ... &  ...  &   ...   \\
0 & 0 &  ... & M^h_L
\end{pmatrix},  \quad
D^h = \begin{pmatrix}
D^h_1 & 0 & ... & 0 \\
0 & D^h_2 & ... &  0 \\
 ... &  ... &  ...  &   ...   \\
0 & 0 &  ... & D^h_L
\end{pmatrix},\quad
F^h = \begin{pmatrix}
F^h_1 \\
F^h_2 \\
 ...   \\
F^h_L
\end{pmatrix},  
\]\[
Q^h = 
\begin{pmatrix}
\sum_{\beta \neq 1} Q^h_{1\beta} & -Q^h_{12} & ...  & -Q^h_{1L} \\
-Q^h_{21} & \sum_{\beta \neq 2} Q^h_{2\beta} & ... & -Q^h_{2L} \\
 ... &  ... &  ...  &   ... &  \\
-Q^h_{L1}   & -Q^h_{L2} & ... & \sum_{\beta \neq L} Q^h_{L \beta}
\end{pmatrix},
\]
and 
\[
A^h = \begin{pmatrix}
A^h_{11} & A^h_{12} & ... & A^h_{1L} \\
A^h_{21} & A^h_{22} & ... & A^h_{2L} \\
 ... &  ... &  ...  &   ...   \\
A^h_{L1} & A^h_{L2} &  ... & A^h_{LL}
\end{pmatrix}
= \begin{pmatrix}
D^h_1 + \sum_{\beta \neq 1} Q^h_{1\beta} & -Q^h_{12} & ...  & -Q^h_{1L} \\
-Q^h_{21} & D^h_2 + \sum_{\beta \neq 2} Q^h_{2\beta} & ... & -Q^h_{2L} \\
 ... &  ... &  ...  &   ... &  \\
-Q^h_{L1}   & -Q^h_{L2} & ... & D^h_L + \sum_{\beta \neq L} Q^h_{L \beta}
\end{pmatrix}
\]
where 
$M^h_{\alpha} = \{m_{\alpha,ij}\}$, 
$D^h_{\alpha} = \{a_{\alpha,ij}\}$, 
$Q^h_{\alpha \beta} = \{q_{\alpha \beta, ij}\}$, 
$F^h_{\alpha} = \{f_{\alpha,j} |\varsigma^{\alpha}_j|\}$
\[
m_{\alpha,ij} =
\left\{\begin{matrix}
c_{\alpha,i} |\varsigma^{\alpha}_i|  & i = j, \\
0 & otherwise
\end{matrix}\right. ,  
\quad
a_{\alpha,ij} =
\left\{\begin{matrix}
\sum_{n \neq i} T_{\alpha, in} & i = j, \\
-T_{\alpha, ij} & otherwise
\end{matrix}\right. ,  
 \quad
q_{\alpha \beta,ij} =
\left\{\begin{matrix}
\sigma_{\alpha \beta,ij} &  \varsigma^{\alpha}_i \cap \varsigma^{\beta}_j \neq 0, \\
0 & otherwise
\end{matrix}\right. .
\]
We note that under assumptions \eqref{mas}, we see that matrices $D^h_{\alpha}$ and $Q^h_{\alpha \beta}$ are symmetric and diagonally dominant with non-negative diagonal entries and therefore are positive semidefinite. Then, the block matrices $A^h$ and $Q^h$ are symmetric and positive semidefinite \cite{golub2013matrix, horn2012matrix}. Moreover, it is well-known that the given finite volume method with two-point flux approximation provides a solution with second-order accuracy by space.

\subsection{Multiscale approximation by space (NLMC)}

A general multiscale finite element approach for multicontinuum upscaling was presented in \cite{chung2017non}. The proposed method is based on the Constrained Energy Minimization Generalized Finite Element Method (CEM-GMsFEM) \cite{chung2017constraint} and is called the nonlocal multicontinua method (NLMC). The MLMC method preserves the physical meaning of the coarse grid approximation by a particular way of defining constraints in the CEM-GMsFEM.
The system obtained using the NLMC method has the same size as a regular finite volume approximation on a coarse grid but provides an accurate approximation with a significant reduction of the discrete system size.


Let $\mathcal{T}_H(\Omega)$ be the coarse grid with cells $K_i$, 
$K^+_i$ be an oversampled region for the coarse cell $K_i$ obtained by enlarging $K_i$ by several coarse cell layers,  
$K_i^{\alpha} =  K_i \cap \Omega_{\alpha}$ and 
$K_i^{\alpha, +} =  K^+_i \cap \Omega_{\alpha}$.
We  construct a set of basis functions $\psi^{i,\alpha} = (\psi^{i,\alpha} _1, \psi^{i,\alpha} _2,...,\psi^{i,\alpha} _L)$ for $\alpha$-continuum in local domain $K_i^+$ with $\psi^{i,\alpha}_{\beta} \in K_i^{\beta,+}$ using the following constrains
\[
\frac{1}{|K^{\beta}_j|}\int_{K^{\beta}_j} \psi^{i,\alpha}_{\beta} dx = \delta_{ij}  \delta_{\alpha \beta},  
\quad \forall K^{\beta}_j \in K^+_i, 
\]
with $\beta = 1,...,L$.
Given constraints provide meaning to the coarse scale solution: the local solution has zero means in another continuum except for the one for which it is formulated.

Then, we solve the following constrained energy minimizing problem in the oversampled local domain ($K_i^+$) using a fine-grid approximation for the coupled system
\begin{equation}
\label{eq:basis}
\begin{pmatrix}
D_1^{K^+_i}+ \sum_{\gamma \neq 1} Q^{K^+_i}_{1\gamma} & ... & -Q^{K^+_i}_{1L} 
& C^T_1 & ... & 0 \\
... & ... & ... & ... & ... & ... \\
-Q^{K^+_i}_{L1} & ...& D^{K^+_i}_L + \sum_{\gamma \neq L} Q^{K^+_i}_{L\gamma} & 0  &  ... &  C^T_L \\
C_1 & ...& 0 & 0  & ...& 0 \\
... & ... & ... & ... & ... & ... \\
0 & ...& C_L &  0 & ...  & 0\\
\end{pmatrix}
\begin{pmatrix}
\psi^{i,\alpha}_1 \\
... \\
\psi^{i,\alpha}_L \\
\mu_1 \\
... \\
\mu_L \\
\end{pmatrix} =
\begin{pmatrix}
0 \\
...\\
0 \\
F^{\alpha,i}_1 \\
...\\
F^{\alpha,i}_L \\
\end{pmatrix}
\end{equation}
with  the zero Dirichlet boundary conditions on $\partial K^+_i$ for $\psi^{i,\alpha}$. 
Here a Lagrange multipliers $\mu_{\beta} = \{\mu^{\beta}_j\}$ are used to impose the constraints with  $F_{\beta}^{\alpha\,i} = \{F^{\alpha,i}_{\beta,j}\}$,  where $F^{\alpha,i}_{\beta,j}$ is  related to the $K^{\beta}_j$ and  $F^{\alpha,i}_{\beta,j} =  \delta_{ij}  \delta_{\alpha \beta}$. 

By combining multiscale basis functions, we obtain the following multiscale space 
\[
V_H = \text{span} \{ \psi^{i,\alpha}=(\psi^{i,\alpha}_1, \psi^{i,\alpha}_2,...,\psi^{i,\alpha}_L), \quad
\alpha = \overline{1,L}, \quad
 i = \overline{1,N_c} \}, 
\]
and form the projection matrix
\[
R = \begin{pmatrix}
R_{11} & R_{12}& ...& R_{1L} \\
R_{21} & R_{22}&  ...& R_{2L} \\
 ...&  ...&  ...&  ... \\
R_{L1} & R_{L2}  &  ...& R_{LL}
\end{pmatrix},
\quad 
R_{\alpha \beta} = \begin{pmatrix}
\psi^{\alpha, 1}_{\beta} \\
\psi^{\alpha, 2}_{\beta} \\
... \\
\psi^{\alpha, N_c}_{\beta} 
\end{pmatrix}.
\]
 
Then, we have the following coarse scale coupled system for $u^H = (u^H_1, u^H_2, ...,u^H_L)$
\begin{equation}
\label{app-nlmc}
M^H \frac{\partial u^H}{\partial t} + A^H u^H  = F^H, 
\end{equation}
where 
\[
D^H = R D^h R^T = \begin{pmatrix}
D^H_{11} & D^H_{12} & ... & D^H_{1L} \\
D^H_{21} & D^H_{22} & ... & D^H_{2L} \\
 ... &  ... &  ...  &   ...   \\
D^H_{L1} & D^H_{L2} &  ... & D^H_{LL}
\end{pmatrix}, \quad 
F^H = \begin{pmatrix}
F^H_1 \\
F^H_2 \\
 ...   \\
F^H_L
\end{pmatrix},  
\]
with $D^H_{\alpha\beta} = \sum_{\gamma} R_{\alpha\gamma} D_{\gamma} R_{\beta\gamma}^T$ and $F^H_{\alpha} = R_{\alpha\alpha} F^h_{\alpha}$.   

Based on the properties of the multiscale basis functions for mass matrices, we have 
\[
M^H = \begin{pmatrix}
M^H_1 & 0 & ... & 0 \\
0 & M^H_2 & ... &  0 \\
 ... &  ... &  ...  &   ...   \\
0 & 0 &  ... & M^H_L
\end{pmatrix},  \quad
Q^H = \begin{pmatrix}
\sum_{\gamma \neq 1} Q^H_{1\gamma} & -Q^H_{12} & ...  & -Q^H_{1L} \\
-Q^H_{21} & \sum_{\gamma \neq 2} Q^H_{2\gamma} & ... & -Q^H_{2L} \\
 ... &  ... &  ...  &   ... &  \\
-Q^H_{L1}   & -Q^H_{L2} & ... & \sum_{\gamma \neq L} Q^H_{L \gamma}
\end{pmatrix},
\]
with $M^H_{\alpha} = R_{\alpha\alpha} M^h_{\alpha} R_{\alpha\alpha}^T$,  $Q^H_{\alpha\beta} = R_{\alpha\alpha} Q^h_{\alpha\beta} R_{\beta\beta}^T$
and 
\[
A^H = \begin{pmatrix}
A^H_{11} & A^H_{12} & ... & A^H_{1L} \\
A^H_{21} & A^H_{22} & ... & A^H_{2L} \\
 ... &  ... &  ...  &   ...   \\
A^H_{L1} & A^H_{L2} &  ... & A^H_{LL}
\end{pmatrix}
= \begin{pmatrix}
D^H_{11} + \sum_{\gamma \neq 1} Q^H_{1\gamma} & D^H_{12} - Q^H_{12} & ...  & D^H_{1L} - Q^H_{1L} \\
D^H_{21} - Q^H_{21} & D^H_{22} + \sum_{\gamma \neq 2} Q^H_{2\gamma} & ... & D^H_{2L} - Q^H_{2L} \\
 ... &  ... &  ...  &   ... &  \\
D^H_{L1} - Q^H_{L1}   & D^H_{L2} - Q^H_{L2} & ... & D^H_{LL} + \sum_{\gamma \neq L} Q^H_{L \gamma}
\end{pmatrix}.
\]

Here for $c_{\alpha}, \sigma_{\alpha \beta}, f_{\alpha} = \const$ in each coarse cell $K_i^{\alpha}$,  the mass matrix, continuum coupling matrix, and right-hand side vector can be directly calculated on the coarse grid and similar to the finite volume approximation on the coarse grid
\[
M^H_{\alpha} = \{{m}_{\alpha,ij}\}, \quad 
Q^H_{\alpha \beta} = \{{q}_{\alpha \beta, ij}\}, \quad 
F^H_{\alpha} = \{{f}_{\alpha,j} |K^{\alpha}_j|\}, \quad 
\]\[
{m}_{\alpha,ij} =
\left\{\begin{matrix}
{c}_{\alpha,i} |K^{\alpha}_i| & i = j, \\
0 & otherwise
\end{matrix}\right. ,  
\quad
q_{\alpha \beta,ij} =
\left\{\begin{matrix}
{\sigma}_{\alpha \beta,ij} &  K^{\alpha}_i \cap K^{\beta}_j \neq 0, \\
0 & otherwise
\end{matrix}\right.
\]
Therefore, we have $M^H=(M^H)^T \geq 0$ and $A^H = (A^H)^T \geq 0$ in the NLMC method.

\section{Implicit approximation by time (coupled scheme)}

We set $u^n = u(t_n)$, where $t_n = n \tau$,  $n=1,2, ...$ and $\tau > 0$ be the uniform time step size. 
To construct time approximations for \eqref{app-mc} and \eqref{app-nlmc} in a general form, we can use the following form of linear multistep scheme
\begin{equation}
\label{mm3}
\frac{1}{\tau} M  \left(\sum_{j=-1}^{s-1} a_j u^{n-j} \right) 
+ A \left(\sum_{j=-1}^{s-1} c_j u^{n-j}\right)  = F, 
\quad s \geq 1
\quad n = 1,2,...
\end{equation}
where $M$ and $A$ are linear operators. 
This system is coupled, and the size of the system is $N_h = \sum_{\alpha} N^{\alpha}_h$ on the fine grid and $N_H = \sum_{\alpha} N^{\alpha}_H$ on the coarse grid.
By Taylor expansion about $t_n$, we can obtain constraints for time approximation schemes of different order \cite{ascher1995implicit}.
Note that the time approximation techniques work with the discrete by-space systems and can be applied to fine and coarse-scale systems.

\subsection{Two-level scheme} 

For two-level  scheme ($s=1$), we have the following constrains $a_{-1} + a_0 = 0$ and $a_{-1} = c_{-1} + c_0$. 
With $a_{-1} = -a_0 = 1$, $c_{-1} = \theta$, $c_0 = (1-\theta)$, we obtain the weighted scheme ($\theta$-scheme)  \cite{samarskii2001theory, vabishchevich2013additive, vabishchevich2012explicit, afanas2013unconditionally, kolesov2014splitting}
\begin{equation}
\label{mm4}
M \frac{u^{n+1} - u^n}{\tau} 
+ A (\theta u^{n+1} + (1-\theta) u^n)  = F,
\end{equation}
where we have the forward and backward Euler schemes for $\theta=0$ and $1$.

\begin{theorem}
\label{t:t1}
The solution of the discrete problem \eqref{mm4}  is stable with $\theta \geq 1/2$ and satisfies the following estimate
\begin{equation}
\label{t1}
||u^{n+1}||_{A}^2 \leq ||u^n||_{A}^2 
+ \frac{\tau}{2} ||F||^2_{M^{-1}}.
\end{equation}
\end{theorem}
\begin{proof}
To find a stability estimate for the two-level  scheme, we write the equation \eqref{mm4} as follows
\[
\left(M + \tau \theta A \right)\frac{u^{n+1} - u^{n}}{\tau} + A u^n = F.
\]
By scalar multiply to $(u^{n+1} - u^{n})/\tau$ with $u^n = (u^{n+1} + u^n)/2 - (u^{n+1} - u^n)/2$, we obtain 
\[
\begin{split}
& \left( (M + \tau \theta A) \frac{u^{n+1} - u^{n}}{\tau}, \frac{u^{n+1} - u^{n}}{\tau} \right) 
 + \left( A u^{n}, \frac{ u^{n+1} - u^{n} }{\tau} \right)\\
&= \left( (M + \tau \theta A) \frac{u^{n+1} - u^{n}}{\tau}, \frac{u^{n+1} - u^{n}}{\tau} \right)  
+ \left( A \frac{u^{n+1} + u^n}{2},   \frac{u^{n+1} - u^{n}}{\tau} \right) 
- \left( A \frac{u^{n+1} - u^n}{2},   \frac{u^{n+1} - u^{n}}{\tau} \right) \\
&= \left( \left(M + \tau \left(\theta - \frac{1}{2} \right) A \right) \frac{u^{n+1} - u^{n}}{\tau}, \frac{u^{n+1} - u^{n}}{\tau} \right)  
+ \left( A \frac{u^{n+1} + u^n}{2},   \frac{u^{n+1} - u^{n}}{\tau} \right) 
= \left(F, \frac{ u^{n+1} - u^{n} }{\tau} \right) 
\end{split}
\]
Here for $A = A^T \geq 0$, we have $( A (u^{n+1} + u^n),  u^{n+1} - u^n) = (A u^{n+1}, u^{n+1}) - (A u^n, u^n)$.

By Cauchy inequality, 
we have
\begin{equation}
\label{csh}
\left(F, \frac{ u^{n+1} - u^{n} }{\tau} \right) 
\leq 
\varepsilon \left\| \frac{ u^{n+1} - u^n}{\tau} \right\|_M^2 
+ \frac{1}{4 \varepsilon} 
\left\| F \right\|_{M^{-1}}^2.
\end{equation}
Therefore the two-level  scheme  is stable with $(\theta - 1/2) \geq 0$ and stability estimate holds.
\end{proof}

\subsection{Three-level scheme} 

For the three-level schemes ($s=2$), we have the following constrains $a_{-1} + a_0 + a_1 = 0$, $a_{-1} - a_1 = c_{-1} + c_0 + c_1$ and $(a_{-1} + a_1)/2 = c_{-1} - c_1$. 
With 
$a_{-1} = \mu$, $a_0 = 1-2 \mu$, $a_1 = \mu-1$, 
$c_{-1}=\mu-1/2+\sigma$, $c_0 = 3/2-\mu-2\sigma$ and $c_1 = \sigma$, we obtain \cite{ascher1995implicit, vabishchevich2013additive, vabishchevich2020explicit}
\begin{equation}
\label{mm5}
\begin{split}
M &\left( \mu \frac{u^{n+1} - u^n}{\tau}  + (1-\mu) \frac{u^{n} - u^{n-1}}{\tau}  \right)\\
&+ A \left( 
\left(\sigma + \mu - \frac{1}{2} \right) u^{n+1} - \left(2 \sigma + \mu - \frac{3}{2} \right) u^n + \sigma u^{n-1}
\right)  
= F.
\end{split}
\end{equation}

\begin{theorem}
\label{t:t2}
The solution of the discrete problem \eqref{mm5}  is stable with $\mu \geq 1/2$ and $\sigma \geq (1-\mu)/2$ and satisfies the following estimate
\begin{equation}
\label{t2}
\frac{1}{4}||u^{n+1} + u^{n}||_{A}^2 + ||u^{n+1} - u^{n}||_{S^{Im}}^2 
\leq 
\frac{1}{4}||u^{n} + u^{n-1}||_{A}^2 + ||u^{n} - u^{n-1}||_{S^{Im}}^2 
+ \frac{\tau}{2} ||F||^2_{M^{-1}}
\end{equation}
with $S^{Im} = \frac{1}{\tau} \left(\mu - \frac{1}{2} \right) M  
+ \left( \sigma + \frac{\mu - 1}{2} \right) A$. 
\end{theorem}
\begin{proof}    
Equation \eqref{mm5} can be written in the following way
\[
\begin{split}
M &\left( \mu \frac{u^{n+1} - u^n}{\tau}  + (1-\mu) \frac{u^{n} - u^{n-1}}{\tau}  \right)
+ A \left( 
\left(\sigma + \mu - \frac{1}{2} \right) (u^{n+1} - u^n) - \sigma (u^n - u^{n-1}) + u^n
\right)  \\
& =
M \left( \mu \frac{u^{n+1} - 2 u^n + u^{n-1}}{\tau}  +  \frac{u^{n} - u^{n-1}}{\tau}  \right)
+ A \left(  
\sigma (u^{n+1} - 2u^n + u^{n-1})
+ \left(\mu - \frac{1}{2} \right) (u^{n+1} - u^n)  + u^n
\right)  \\
& =
 \frac{1}{\tau} \underbrace{\left(\mu M  + \tau \sigma A \right) (u^{n+1} - 2 u^n + u^{n-1})}_{\text{I}}
+ \frac{1}{\tau} \underbrace{M (u^n - u^{n-1})}_{\text{II}}
+ \left(\mu - \frac{1}{2} \right) \underbrace{A (u^{n+1} - u^n)}_{\text{III}}  
+ \underbrace{A u^n}_{\text{IV}} 
= F.
\end{split}
\]
Let $w^{n} = u^n - u^{n-1},\quad 2 y^n= u^{n} + u^{n-1}$, then
\[
w^{n+1} - w^n = u^{n+1} - 2 u^n + u^{n-1}, \quad 
w^{n+1} + w^n = u^{n+1} - u^{n-1} = 2 (y^{n+1} - y^n), 
\]\[
w^n = (w^{n+1} + w^n)/2 - (w^{n+1} - w^n)/2, \quad 
w^{n+1} = (w^{n+1} + w^n)/2 + (w^{n+1} - w^n)/2,
\]
and 
\[
u^n = (y^{n+1} + y^n)/2 - (w^{n+1} - w^n)/4.
\]
Next, we multiply to $w^{n+1} + w^n$ (in a scalar way) and obtain
\[
\text{I:} \quad ( ( \mu M  + \tau \sigma A ) (w^{n+1} - w^n), w^{n+1} + w^n) 
= ( ( \mu M  + \tau \sigma A )w^{n+1}, w^{n+1} )  
- ( ( \mu M  + \tau \sigma A ) w^n, w^n),
\]\[
\begin{split}
\text{II:} \quad (M w^n, w^{n+1} + w^n ) 
& = \frac{1}{2} ( M (w^{n+1} + w^n), w^{n+1} + w^n)  
- \frac{1}{2} ( M (w^{n+1} - w^n), w^{n+1} + w^n) \\
& = \frac{1}{2} ( M (w^{n+1} + w^n), w^{n+1} + w^n)  
- \frac{1}{2}(M w^{n+1}, w^{n+1}) + \frac{1}{2}(M w^n, w^n),
\end{split}
\]\[
\begin{split}
\text{III:} \quad (A w^{n+1}, w^{n+1} + w^n ) 
& = \frac{1}{2} ( A (w^{n+1} + w^n), w^{n+1} + w^n)  
+ \frac{1}{2} ( A (w^{n+1} - w^n), w^{n+1} + w^n) \\
& = \frac{1}{2} ( A (w^{n+1} + w^n), w^{n+1} + w^n)  
+ \frac{1}{2} (A w^{n+1}, w^{n+1}) - \frac{1}{2}(A w^n, w^n),
\end{split}
\]\[
\begin{split}
\text{IV:} \quad (A u^n, w^{n+1} + w^n ) 
& = ( A (y^{n+1} + y^n), (y^{n+1} - y^n))  
- \frac{1}{4} ( A (w^{n+1} - w^n), w^{n+1} + w^n) \\
& = (A y^{n+1}, y^{n+1}) - (A y^n, y^n)
- \frac{1}{4}(A w^{n+1}, w^{n+1}) + \frac{1}{4} (A w^n, w^n),
\end{split}
\]
for $A=A^T$ and $M=M^T$.

Then, we have
\[
\begin{split}
 \frac{1}{2} & \left( 
\left( \frac{1}{\tau} M + \left(\mu - \frac{1}{2} \right) A
\right) (w^{n+1} + w^n), w^{n+1} + w^n \right)\\
& + \left( 
\left( \frac{\mu}{\tau} M  + \sigma A - \frac{1}{2 \tau} M  + \left( \frac{\mu}{2} - \frac{1}{4} \right) A - \frac{1}{4} A 
\right) w^{n+1}, w^{n+1} \right)  + (A y^{n+1}, y^{n+1})\\
& -
\left( 
\left( \frac{\mu}{\tau} M  + \sigma A - \frac{1}{2 \tau} M  + \left( \frac{\mu}{2} - \frac{1}{4} \right) A - \frac{1}{4} A 
\right) w^{n}, w^{n} \right)  - (A y^n, y^n)= (F, w^{n+1} + w^n).
\end{split}
\]
Then using Cauchy inequality \eqref{csh}, we have 
\[
\begin{split}
 \frac{1}{2} & 
\left(\mu - \frac{1}{2} \right) \left( A (w^{n+1} + w^n), w^{n+1} + w^n \right)\\
&+ 
\left( 
\left( 
\frac{1}{\tau} \left(\mu - \frac{1}{2} \right) M  
+ \left( \sigma + \frac{\mu - 1}{2} \right) A \right) w^{n+1}, w^{n+1} \right)  
+ (A y^{n+1}, y^{n+1})\\
& -
\left( 
\left(\frac{1}{\tau} \left(\mu - \frac{1}{2} \right) M  
+ \left( \sigma + \frac{\mu - 1}{2} \right) A \right) w^{n}, w^{n} \right)  
- (A y^n, y^n) 
\leq 
\frac{\tau}{2}  \left\| F \right\|_{M^{-1}}^2.
\end{split}
\]
Let $S^{Im} = \frac{1}{\tau} \left(\mu - \frac{1}{2} \right) M  
+ \left( \sigma + \frac{\mu - 1}{2} \right) A$ and for $\mu \geq 1/2$, we have
\[
(S^{Im} w^{n+1}, w^{n+1}) - (A y^{n+1}, y^{n+1})\\
\leq 
(S^{Im} w^{n}, w^{n}) + (A y^n, y^n) 
+
\frac{\tau}{2}  \left\| F \right\|_{M^{-1}}^2.
\]
Then
\[
\begin{split}
(S^{Im} (u^{n+1} - u^n), u^{n+1} - u^n) 
&+ \frac{1}{4}(A (u^{n+1} + u^n), u^{n+1} + u^n)\\
\leq 
&(S^{Im} (u^n - u^{n-1}), u^n - u^{n-1}) 
+ \frac{1}{4}(A (u^{n} + u^{n-1}), u^{n} + u^{n-1}) 
+ \frac{\tau}{2}  \left\| F \right\|_{M^{-1}}^2.
\end{split}
\]
Therefore, the three-level  scheme is stable with $\mu \geq 1/2$ and $\sigma \geq (1-\mu)/2$ and stability estimate holds.
\end{proof}

For example, with $\mu = 1/2$, we obtain
\[
M \frac{u^{n+1} - u^{n-1}}{2 \tau}
+ A \left( 
\sigma u^{n+1} + (1 - 2 \sigma) u^n + \sigma u^{n-1}
\right)  = F,
\]
where with $\sigma=0$ we have LF (Leap Frog). 
A similar three-level scheme with weights was presented in \cite{vabishchevich2013additive} with $\sigma = \sigma_1 = \sigma_2$
\[
M \frac{u^{n+1} - u^{n-1}}{2 \tau}
+ A \left( 
\sigma_1 u^{n+1} + (1 - \sigma_1 - \sigma_2) u^n + \sigma_2 u^{n-1}
\right)  = F,
\]
with stability for $\sigma_1 \geq \sigma_2$ and $\sigma_1 + \sigma_2 > 1/2$.   

For $\mu=1$, we have
\[
M \frac{u^{n+1} - u^n}{\tau} 
+ A \left( 
\left( \sigma + \frac{1}{2} \right) u^{n+1} - \left( 2 \sigma - \frac{1}{2} \right) u^n + \sigma u^{n-1}
\right)  = F,
\]
where with $\sigma=0$ we have CN (Crank-Nicolson).

For $\mu=3/2$, we have
\[
M  \frac{3 u^{n+1} - 4 u^n + u^{n-1}}{2 \tau} 
+ A \left( 
(\sigma + 1) u^{n+1} - 2 \sigma u^n + \sigma u^{n-1}
\right)  = F,
\]
where with $\sigma=0$ we have BDF2 (Backward Differentiation Formula).

Presented implicit approximation leads to the large coupled system of equations with high contrast properties. To separate equations by contrasts, we decouple the system of equations by constructing an Implicit-Explicit time approximation. The special additive representation of the operator with an explicit time approximation of the coupling term will lead us to the solution of the smaller problems associated with each continuum.

\section{Implicit-Explicit scheme (decoupled scheme)}

Let $M$ be a diagonal matrix and $A = A^{(1)} + A^{(2)}$ with $A^{(1)} = \text{diag}(A_{11}, A_{22}, ..., A_{LL})$ and $A^{(2)} = A - A^{(1)}$. 
Such additive representation helps us to separate continuum coupling terms and, by approximating it explicitly, construct decoupled schemes
\begin{equation}
\label{cm4}
\underbrace{
\frac{1}{\tau} M  (\sum_{j=-1}^{s-1} a_j u^{n-j})
+A^{(1)} (\sum_{j=-1}^{s-1} c_j u^{n-j})
}_{\text{Im}} 
+ \underbrace{
A^{(2)} (\sum_{j=0}^{s-1} d_j u^{n-j})
}_{\text{Ex}} 
= F.
\end{equation} 
Note that we will have additional coupling for the non-diagonal matrix $M$ and an additive representation should also be utilized to matrix $M$. The properties of the matrix $M$ depend on the approximation method; for example, in the finite element method, we will have a non-diagonal matrix. This work considers finite volume approximation and the NLMC method, which lead to the diagonal matrix $M$. 

In considered multicontinuum problem  \eqref{app-mc}, we have  (\textit{D-scheme})
\begin{equation}
\label{cmD}
A^{(1)} = \begin{pmatrix}
A_{11} & 0 & ... & 0 \\
0 & A_{22} & ... &  0 \\
 ... &  ... &  ...  &   ...   \\
0 & 0 &  ... & A_{LL}
\end{pmatrix},
\quad 
A^{(2)} = A - A^{(1)} = 
 \begin{pmatrix}
0 & A_{12} & ...  & A_{1L}\\
A_{21} &0 & ... & A_{2L} \\
 ... &  ... &  ...  &   ...   \\
A_{L1}   & A_{L2} & ... & 0
\end{pmatrix}.
\end{equation}
We have the following properties for operators 
\begin{equation}
\label{dc-op}
A = A^{(1)} + A^{(2)}  \geq 0, \quad
A^{(1)} - A^{(2)} \geq 0,
\end{equation}
where the property of operator $A^{(1)}-A^{(2)}$ follows from the diagonally dominance with non-negative diagonal entries \cite{gaspar2014explicit, vasilyeva2023efficient}. 

Furthermore, the additive representation of the operator $A$ with diagonal operator $A^{(1)}$ is similar to the Jacobi iterations \cite{golub2013matrix}. Therefore, the most straightforward improvement of the method can be made by incorporating the available information of the solution of continuum $\alpha$ in computing solution for continuum $\alpha+1$ like done in Gauss-Seidel iterations \cite{vasilyeva2023efficient}. 
We sort the continuum based on their permeability,  $k_1 < k_2 <...<k_L$, and use the calculated continuum solution in the solution of the following equation. 
Therefore, we can construct the following modified schemes.
\begin{itemize}
\item (\textit{L-scheme}): We first calculate continuum with smaller permeability
\begin{equation}
\label{cmL}
A^{(1)} = 
\begin{pmatrix}
A_{11} & 0 & ... & 0 \\
A_{21} & A_{22} & ... &  0 \\
 ... &  ... &  ...  &   ...   \\
A_{L1} & A_{L2} &  ... & A_{LL}
\end{pmatrix}, \quad 
A^{(2)} = A - A^{(1)} = 
\begin{pmatrix}
0 & A_{12} & ... & A_{1L} \\
0 & 0 & ... &  A_{2L} \\
 ... &  ... &  ...  &   ...   \\
0 & 0 &  ... & 0
\end{pmatrix},
\end{equation}
where $\bar{A}_0$ is the lower-triangular matrix and we have a forward substitution. 

\item (\textit{U-scheme}) We first calculate continuum with larger permeability
\begin{equation}
\label{cmU}
A^{(1)} =
\begin{pmatrix}
A_{11} & A_{12} & ... & A_{1L} \\
0 & A_{22} & ... &  A_{2L} \\
 ... &  ... &  ...  &   ...   \\
0 & 0 &  ... & A_{LL}
\end{pmatrix}, \quad 
A^{(2)} =  A - A^{(1)} = 
\begin{pmatrix}
0 & 0 & ... & 0 \\
A_{21} & 0 & ... &  0 \\
 ... &  ... &  ...  &   ...   \\
A_{L1} & A_{L2} &  ... & 0
\end{pmatrix},
\end{equation}
where $\bar{A}_0$ is the upper-triangular matrix and we have a backward substitution. 
\end{itemize}
The resulting system is decoupled, but we utilize information about the calculated solution for the previous continuum without affecting the solution time.

\subsection{Two-level scheme} 

For two-level  scheme ($s=1$), we have the following constrains $a_{-1} + a_0 = 0$ and $a_{-1} = c_{-1} + c_0 = d_0$. 
With 
$a_{-1} = -a_0 = 1$, 
$c_{-1} = \theta$, $c_0 = (1-\theta)$ and $d_0 = 1$, we obtain the weighted  Implicit-Explicit  scheme
\begin{equation}
\label{cm5}
M \frac{u^{n+1} - u^n}{\tau} 
+ A^{(1)} (\theta u^{n+1} + (1-\theta) u^n)  
+ A^{(2)} u^n = F.
\end{equation}

\begin{theorem}
\label{t:t3}
The solution of the discrete problem \eqref{cm5} is stable with 
\[
\left(\theta - \frac{1}{2}\right) A^{(1)} \geq \frac{1}2 A^{(2)}
\]
and satisfies the following estimate
\begin{equation}
\label{t3}
||u^{n+1}||_{A}^2 \leq ||u^{n}||_{A}^2 
+\frac{\tau}{2} ||F||^2_{\rev{ M^{-1}}}.
\end{equation}
\end{theorem}
\begin{proof}
 To find stability estimate for the two-level  scheme, we write the equation \eqref{mm4} as follows
\[
\left(M + \tau \theta A^{(1)} \right)\frac{u^{n+1} - u^{n}}{\tau} + A u^n = F.
\]
Similarly to the Theorem \ref{t:t1} for $A = A^T \geq 0$, we have
\[
\begin{split}
& \left( (M + \tau \theta A^{(1)}) \frac{u^{n+1} - u^{n}}{\tau}, \frac{u^{n+1} - u^{n}}{\tau} \right) 
 + \left( A u^{n}, \frac{ u^{n+1} - u^{n} }{\tau} \right)\\
&= \left( \left(M + \tau \left(\theta A^{(1)} - \frac{1}{2} A \right) \right) \frac{u^{n+1} - u^{n}}{\tau}, \frac{u^{n+1} - u^{n}}{\tau} \right)  
+ \frac{1}{2 \tau} ||u^{n+1}||_{A}^2
- \frac{1}{2 \tau} ||u^{n}||_{A}^2
= \left(F, \frac{ u^{n+1} - u^{n} }{\tau} \right) 
\end{split}
\]
Finally, by Cauchy inequality \eqref{csh} and for $\left( \theta A^{(1)} - \frac{1}{2} A \right) \geq 0$, the two-level  scheme  is stable and stability estimate holds.
\end{proof}

We note that the scheme is similar to the one presented in \cite{vabishchevich2020explicit}. However, in this paper we present a general way of the implicit-explicit scheme construction as a multistep implicit-explicit scheme presented in \cite{ascher1995implicit}.

\subsection{Three-level scheme} 

For the three-level  schemes ($s=2$), we have 
$a_{-1} + a_0 + a_1 = 0$, 
$a_{-1} - a_1 = c_{-1} + c_0 + c_1 = = d_0 + d_1$ and 
$(a_{-1} + a_1)/2 =  c_{-1} - c_1 = -d_1$. 
With 
$a_{-1} = \mu$, $a_0 = 1-2 \mu$, $a_1 = \mu-1$, 
$c_{-1}=\mu-1/2+\sigma$, $c_0 = 3/2-\mu-2\sigma$, $c_1 = \sigma$, 
$d_0 = 1/2+\mu$, $d_1 = 1/2-\mu$, 
we obtain
\begin{equation}
\label{cm6}
\begin{split}
M & \left( \mu \frac{u^{n+1} - u^n}{\tau}  + (1-\mu) \frac{u^{n} - u^{n-1}}{\tau}  \right)\\
&+ A^{(1)} \left( 
\left(\sigma + \mu - \frac{1}{2} \right) u^{n+1} - \left(2 \sigma + \mu - \frac{3}{2} \right) u^n + \sigma u^{n-1}
\right) \\
&+ A^{(2)} \left( 
\left(\mu + \frac{1}{2} \right) u^n -  \left(\mu - \frac{1}{2} \right) u^{n-1}
\right)
= F.
\end{split}
\end{equation}

\begin{theorem}
\label{t:t4}
The solution of the discrete problem \eqref{cm6}  is stable with 
\[
\mu \geq 1/2, \quad 
\left( \sigma + \frac{\mu - 1}{2} \right) A^{(1)} \geq \frac{\mu}{2} A^{(2)}
\]
and satisfies the following estimate
\begin{equation}
\label{t4}
\frac{1}{4}||u^{n+1} + u^{n}||_{A}^2 + ||u^{n+1} - u^{n}||_{S^{ImEx}}^2 
\leq 
\frac{1}{4}||u^{n} + u^{n-1}||_{A}^2 + ||u^{n} - u^{n-1}||_{S^{ImEx}}^2 
+ \frac{\tau}{2} ||F||^2_{M^{-1}}
\end{equation}
with $S^{ImEx} = \frac{1}{\tau} \left(\mu - \frac{1}{2} \right) M  
+ \left( \sigma + \frac{\mu - 1}{2} \right) A^{(1)} - \frac{\mu}{2} A^{(2)}$. 
\end{theorem}
\begin{proof}
Equation \eqref{cm6} can be written in the following way
\[
\frac{1}{\tau}  
\underbrace{A_{\text{I}} (u^{n+1} - 2 u^n + u^{n-1})}_{\text{I}}
+ \frac{1}{\tau}  
\underbrace{A_{\text{II}} (u^n - u^{n-1})}_{\text{II}}
+ \left(\mu - \frac{1}{2} \right) 
\underbrace{ A_{\text{III}}(u^{n+1} - u^n)}_{\text{III}}  
+ \underbrace{ A_{\text{IV}} u^n}_{\text{IV}} 
= F.
\]
with 
\[
A_{\text{I}} =  \left( \mu M  + \tau \sigma A^{(1)} \right),  \quad 
A_{\text{II}} = \left( M + \tau  \left(\mu -\frac{1}{2} \right) A^{(2)}\right), \quad 
A_{\text{III}} = A^{(1)}, \quad 
A_{\text{IV}} = A^{(1)} + A^{(2)} = A.
\]
Next, we do a scalar multiply to $w^{n+1} + w^n$ ($w^{n} = u^n - u^{n-1},\quad 2 y^n= u^{n} + u^{n-1}$) and obtain
\[
\begin{split}
&\text{I:} \quad ( A_{\text{I}} (w^{n+1} - w^n), w^{n+1} + w^n) 
= (A_{\text{I}} w^{n+1}, w^{n+1} )  
- (A_{\text{I}} w^n, w^n),  
\\
&\text{II:} \quad (A_{\text{II}} w^n, w^{n+1} + w^n ) 
 = \frac{1}{2} ( A_{\text{II}} (w^{n+1} + w^n), w^{n+1} + w^n)  
- \frac{1}{2}(A_{\text{II}} w^{n+1}, w^{n+1}) 
+ \frac{1}{2}(A_{\text{II}} w^n, w^n),
\\
&\text{III:} \quad (A_{\text{III}} w^{n+1}, w^{n+1} + w^n ) 
 = \frac{1}{2} (A_{\text{III}} (w^{n+1} + w^n), w^{n+1} + w^n)  
+ \frac{1}{2} (A_{\text{III}} w^{n+1}, w^{n+1})
- \frac{1}{2}(A_{\text{III}} w^n, w^n),
\\
&\text{IV:} \quad (A_{\text{IV}}  u^n, w^{n+1} + w^n ) 
 = (A_{\text{IV}}  y^{n+1}, y^{n+1}) 
 - (A_{\text{IV}} y^n, y^n)
- \frac{1}{4}(A_{\text{IV}} w^{n+1}, w^{n+1})
+ \frac{1}{4} (A_{\text{IV}} w^n, w^n).
\end{split}
\]
Then, we have
\[
\begin{split}
 \frac{1}{2} & \left( 
\left( 
\frac{1}{\tau} A_{\text{II}} + \left(\mu - \frac{1}{2} \right) A_{\text{III}} 
\right) (w^{n+1} + w^n), w^{n+1} + w^n \right)\\
& + \left( 
\left( \frac{1}{\tau}  A_{\text{I}} - \frac{1}{2 \tau} A_{\text{II}}  
+ \left( \frac{\mu}{2} - \frac{1}{4} \right) A_{\text{III}} 
- \frac{1}{4} A_{\text{IV}}
\right) w^{n+1}, w^{n+1} \right)  
+ (A_{\text{IV}} y^{n+1}, y^{n+1})\\
& -
\left( 
\left( \frac{1}{\tau}  A_{\text{I}}  - \frac{1}{2 \tau} A_{\text{II}}  
+ \left( \frac{\mu}{2} - \frac{1}{4} \right) A_{\text{III}} 
- \frac{1}{4} A_{\text{IV}}
\right) w^{n}, w^{n} \right)  - (A_{\text{IV}} y^n, y^n)
= (F, w^{n+1} + w^n)
\end{split}
\]
Here
\[
\frac{1}{\tau} A_{\text{II}} + \left(\mu - \frac{1}{2} \right) A_{\text{III}}  
= \frac{1}{\tau} \left( M + \tau  \left(\mu - \frac{1}{2} \right) A^{(2)} \right)
+ \left(\mu - \frac{1}{2} \right)  A^{(1)}
= \frac{1}{\tau} M + \left(\mu - \frac{1}{2} \right) A,
\]
and 
\[\begin{split}
S^{ImEx} &= 
\frac{1}{\tau}  A_{\text{I}}  - \frac{1}{2 \tau} A_{\text{II}}  
+ \left( \frac{\mu}{2} - \frac{1}{4} \right) A_{\text{III}} 
- \frac{1}{4} A_{\text{IV}} 
\\
&= \frac{1}{\tau}  \left( \mu M  + \tau \sigma A^{(1)} \right)
 - \frac{1}{2 \tau}  \left( M + \tau \left(\mu - \frac{1}{2} \right) A^{(2)} \right)
+ \left( \frac{\mu}{2} - \frac{1}{4} \right) A^{(1)} 
- \frac{1}{4} A
\\
&= \frac{1}{\tau}  \left( \mu - \frac{1}{2} \right) M
+ \left( \sigma + \frac{\mu-1}{2} \right) A^{(1)} 
- \frac{\mu}{2} A^{(2)},  
\end{split}
\]
then using inequality \eqref{csh} we obtain a stable scheme for $\mu \geq 1/2$ and 
$\left( \sigma + \frac{\mu - 1}{2} \right) A^{(1)} \geq \frac{\mu}{2} A^{(2)}$.
\end{proof}   

For the particular choice of the parameter $\mu \geq 1/2$, we can obtain the following schemes:
\begin{itemize}
\item 
$\mu = 1/2$:
\[
M \frac{u^{n+1} - u^{n-1}}{2 \tau}
+ A^{(1)} \left( 
\sigma u^{n+1} + (1 - 2 \sigma) u^n + \sigma u^{n-1}
\right) + A^{(2)} u^n = F,
\]
where with $\sigma=0$ we have LF (Leap Frog), and with $\sigma = \frac{1}{2}$ we have CN (Crank-Nicolson) for the implicit part and LF for the explicit part (CNLF) \cite{ascher1995implicit}. 
A similar three-level scheme with weights was presented in \cite{vabishchevich2020explicit} that is stable with $(4 \sigma -1) A^{(1)} \geq A^{(2)}$.

\item $\mu=1$:
\[
M \frac{u^{n+1} - u^n}{\tau} 
+ A^{(1)} \left( 
\left( \sigma + \frac{1}{2} \right) u^{n+1} - \left( 2 \sigma - \frac{1}{2} \right) u^n + \sigma u^{n-1}
\right) 
+ A^{(2)} \left( \frac{3}{2} u^n - \frac{1}{2} u^{n-1} \right)= F,
\]
where with $\sigma=0$, we have CN for the implicit part and AB (Adams-Bashforth) for the explicit part (CNAB) \cite{ascher1995implicit}.

\item $\mu=3/2$:
\[
M  \frac{3 u^{n+1} - 4 u^n + u^{n-1}}{2 \tau} 
+ A^{(1)} \left( 
(\sigma + 1) u^{n+1} - 2 \sigma u^n + \sigma u^{n-1}
\right) 
+ A^{(2)} (2 u^n - u^{n-1})= F,
\]
where with $\sigma=0$ we have BDF2 for implicit part (semi-implicit BDF, SBDF) \cite{ascher1995implicit, vabishchevich2020explicit}.
\end{itemize}

\section{Numerical results}

We investigate two test problems with two and three-continuum in fractured porous media, where one of the continuums represents hydraulic fracture. Let $\gamma \subset \mathbb{R}^{d-1}$ be the lower dimensional domain for fractures and $\Omega \in \mathbb{R}^{d}$ be the domain for a porous matrix with $d = 2$. 
To model fractured porous media, we use a mixed dimensional formulation \cite{martin2005modeling, d2012mixed, formaggia2014reduced, Quarteroni2008coupling, vasilyeva2019nonlocal}. Let $u_{\alpha}$ be the $\alpha$-continuum pressure defined in domain $\Omega$ and $u_f$ be the fracture continuum pressure in lower-dimensional domain $\gamma$.

We study two cases:
\begin{itemize}
\item Two-continuum media (\textit{2C}). 

We consider porous matrix continuum ($u_m \in \Omega$) and fracture continuum ($u_f \in \gamma$)  
\[
\begin{split}
& c_m \frac{ \partial u_m }{\partial t} 
- \nabla \cdot (k_m \nabla u_m) +  \sigma_{mf} (u_m - u_f) =   f_m, 
\quad  x \in \Omega, \\
& c_f \frac{ \partial u_f }{\partial t} 
- \nabla \cdot (k_f \nabla u_f) +  \sigma_{mf} (u_f - u_m)=   f_f.
\quad  x \in \gamma.
\end{split}
\]
where we set $c_f = 1$ and $k_f = 10^6$ for fracture continuum,  $c_m = 0.1$ and $k_m = 1$  for porous matrix continuum.  

\item Three-continuum media (\textit{3C}). 

We have porous matrix continuum ($u_1 \in \Omega$), natural fracture continuum ($u_2 \in \Omega$), and hydraulic fracture continuum  ($u_f \in \gamma$)  
\[
\begin{split}
& c_{1} \frac{ \partial u_1}{\partial t}
- \nabla \cdot ( k_{1}  \nabla u_{1})
+   \sigma_{1 2} (u_1 - u_2)
+   \sigma_{1 f} (u_1 - u_f) = f_1, \quad
x \in \Omega, \\
& c_{2} \frac{ \partial u_2}{\partial t}
- \nabla \cdot (k_{2}  \nabla u_{2})
+   \sigma_{21} (u_2 - u_1)
+   \sigma_{2f} (u_2 - u_f)  = f_2, \quad
x \in \Omega, \\
& c_{f} \frac{ \partial  u_f}{\partial t}
- \nabla \cdot (k_{f}  \nabla u_{f})
+   \sigma_{f1} (u_f - u_1)
+   \sigma_{f2} (u_f - u_2)  = f_f, \quad
x \in \gamma,
\end{split}
\]
where we set $c_f = 1$ and $k_f = 10^6$ for hydraulic fracture continuum,  $c_2 = 0.1$ and $k_2 = 1$ for natural fracture continuum,  and  $c_1 = 0.05$ and $k_1 = 10^{-3}$ for porous matrix continuum.  
\end{itemize}

The problem is considered in domain $\Omega = [0,2] \times [0,1]$ with 10 fracture lines $\gamma_l$, $\gamma = \cup_{l=1}^{10} \gamma_l$ (Figure \ref{fig:geom}). 
For domain $\Omega$, the fine grid is $400 \times 200$ grid with $N_h = 80,000$ square cells ($h = 1/200$) and the coarse grid is $40 \times 20$ grid with $N_H = 800$ square cells ($H = 1/20$). 
For the lower-dimensional fracture domain $\gamma$, the fine grid contains $1,474$ cells, and the coarse grid contains $156$ cells. 
Lower-dimensional fractures are considered an overlaying continuum and approximated using embedded fracture model (EFM) \cite{hkj12, ctene2016algebraic, tene2016multiscale}. The fine-scale approach is accurate for a sufficiently fine grid. We construct accurate multiscale basis functions in the NLMC method for coarse-scale approximation to incorporate complex flow behavior on a coarse grid model.

\begin{figure}[h!]
\centering
\includegraphics[width=0.6 \textwidth]{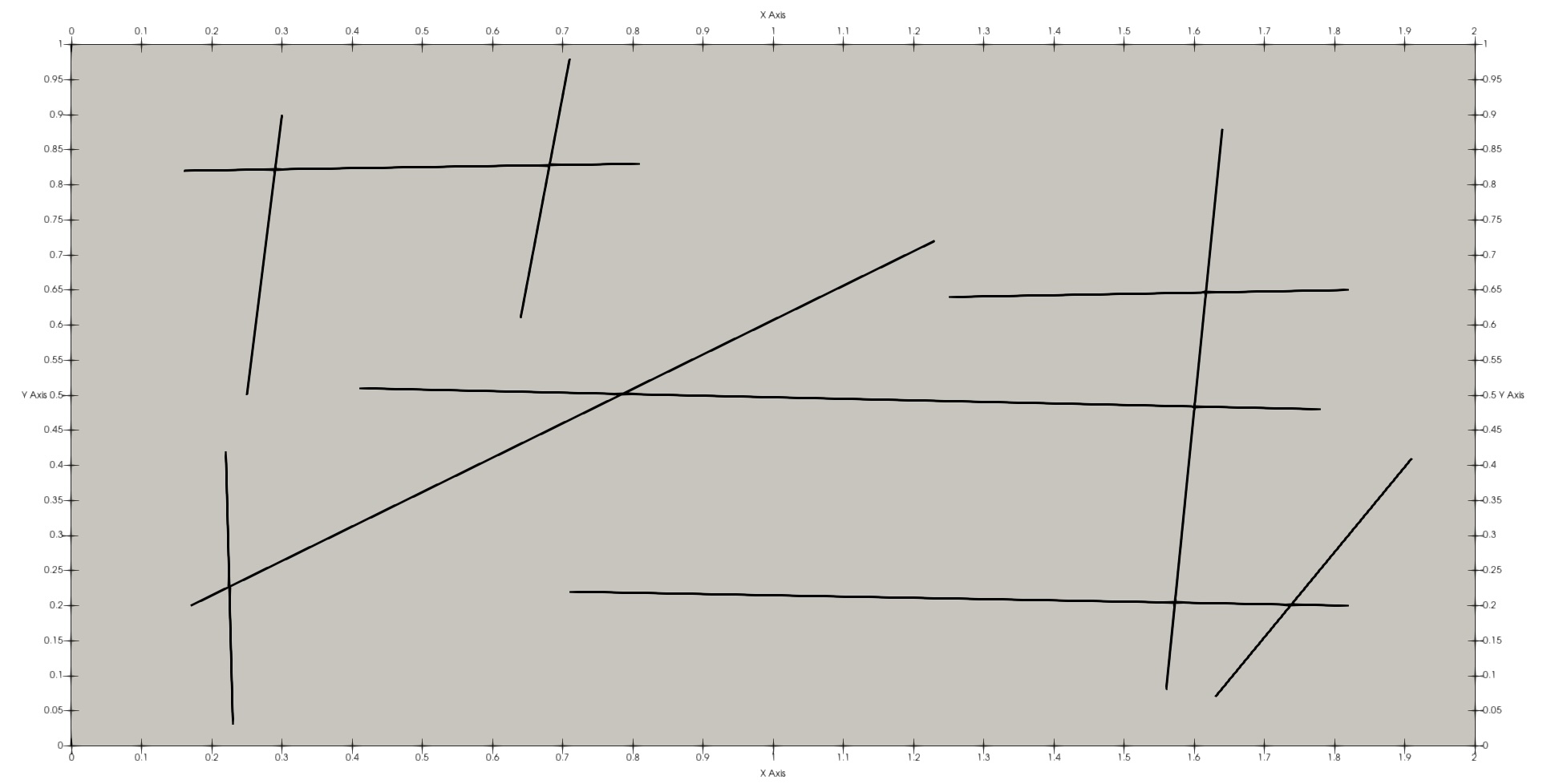}
\caption{Computational domain $\Omega = [0,2] \times [0,1]$ with 10 fractures $\gamma_l$, $\gamma = \cup_{l=1}^{10} \gamma_l$}
\label{fig:geom}
\end{figure}

We set source term in fractures located in the coarse fracture cell related to the domain $[1.85, 1.9] \times [0.35, 0.4]$).  We set $f_m=0$, $f_1=f_2=0$ and $f_f(x) = q_w(x) (u_f - u_w)$ for $x \in \gamma$ with $u_w = 1.2$ and $q_w = 10^5$ and zero otherwise.  
As initial condition, we set $u_{\alpha,0} = 1$ ($\alpha = m, 1,2,f$) and perform simulations with $T_{max} = 0.005$. 

\begin{figure}[h!]
\centering
\includegraphics[width=0.32 \textwidth]{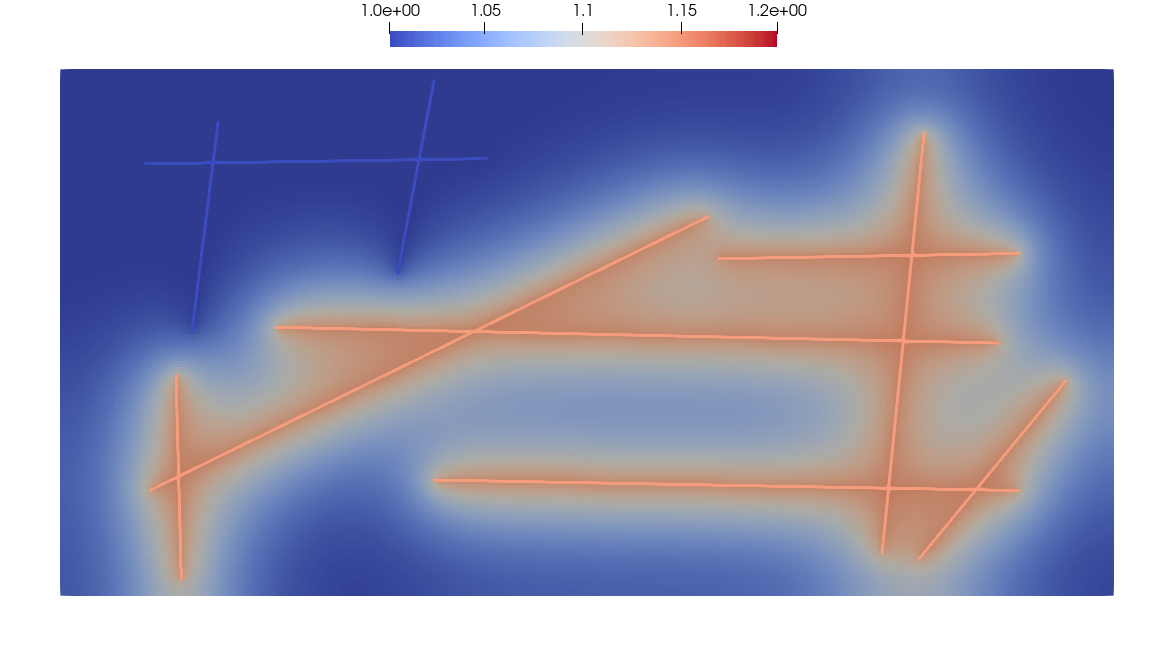}
\includegraphics[width=0.32 \textwidth]{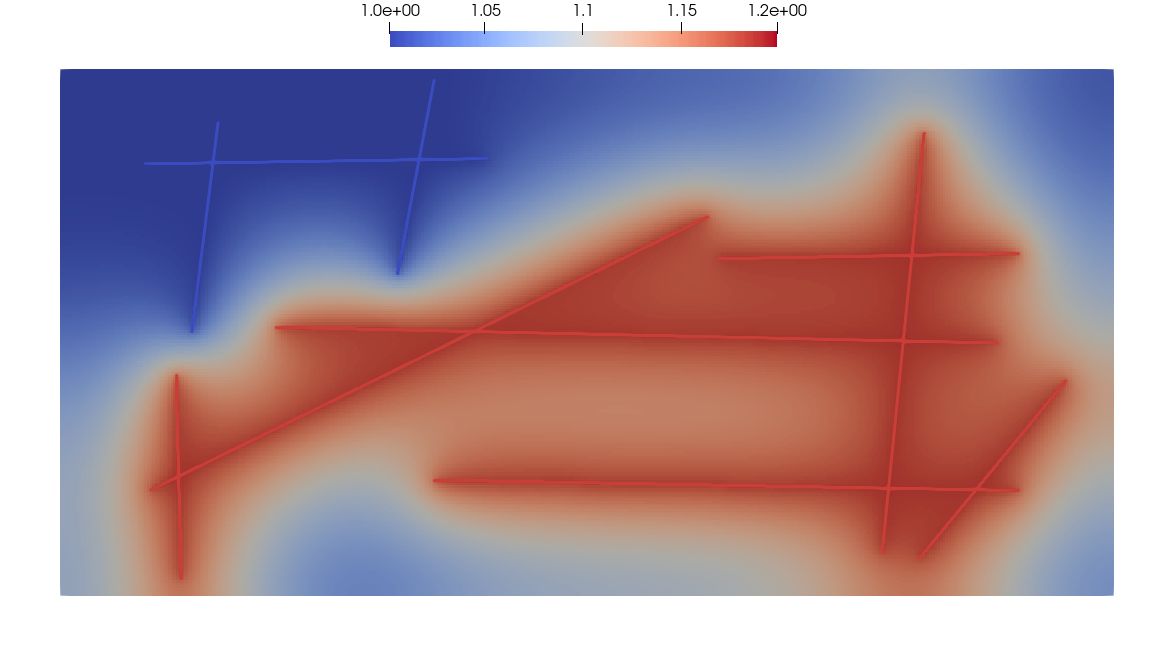}
\includegraphics[width=0.32 \textwidth]{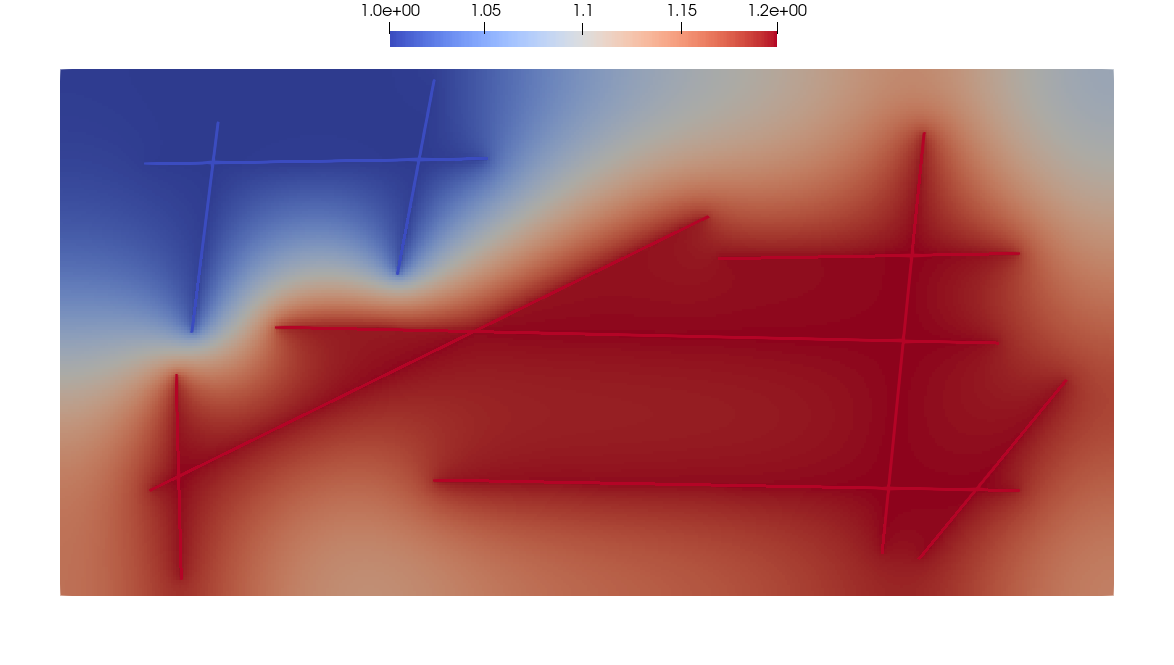}
\caption{Two-continuum media (\textit{2C}). Reference solution at $t = T_{max}/4, T_{max}/2$ and $T_{max}$ (from left to right)}
\label{fig:u2-f}
\end{figure}

\begin{figure}[h!]
\centering
\includegraphics[width=0.32 \textwidth]{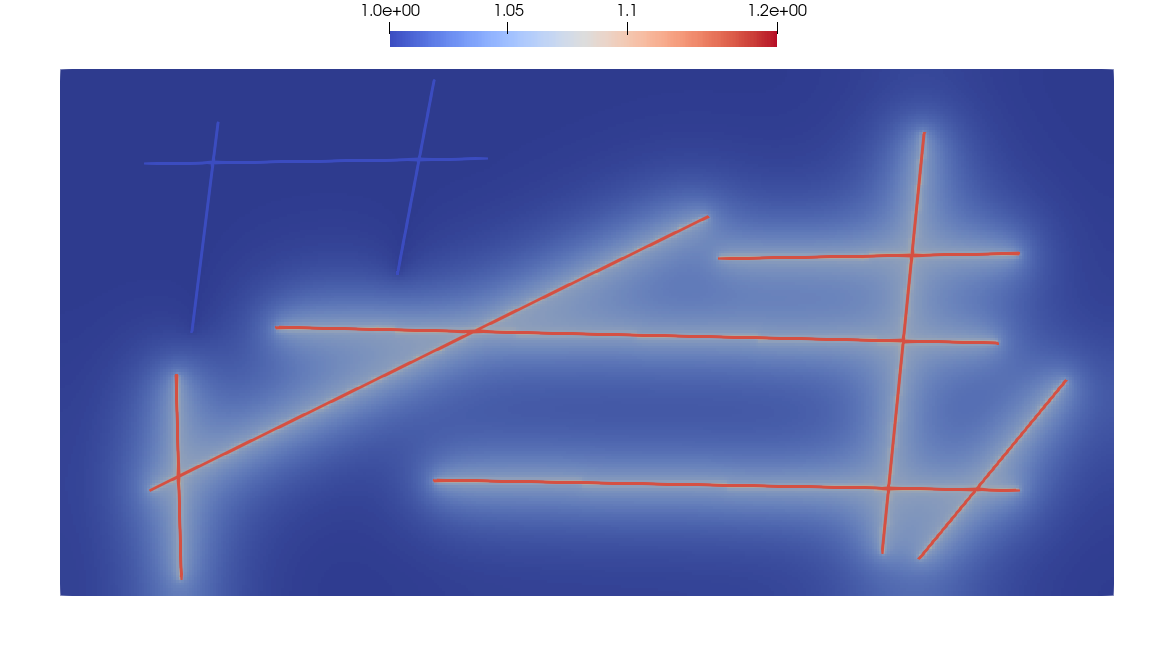}
\includegraphics[width=0.32 \textwidth]{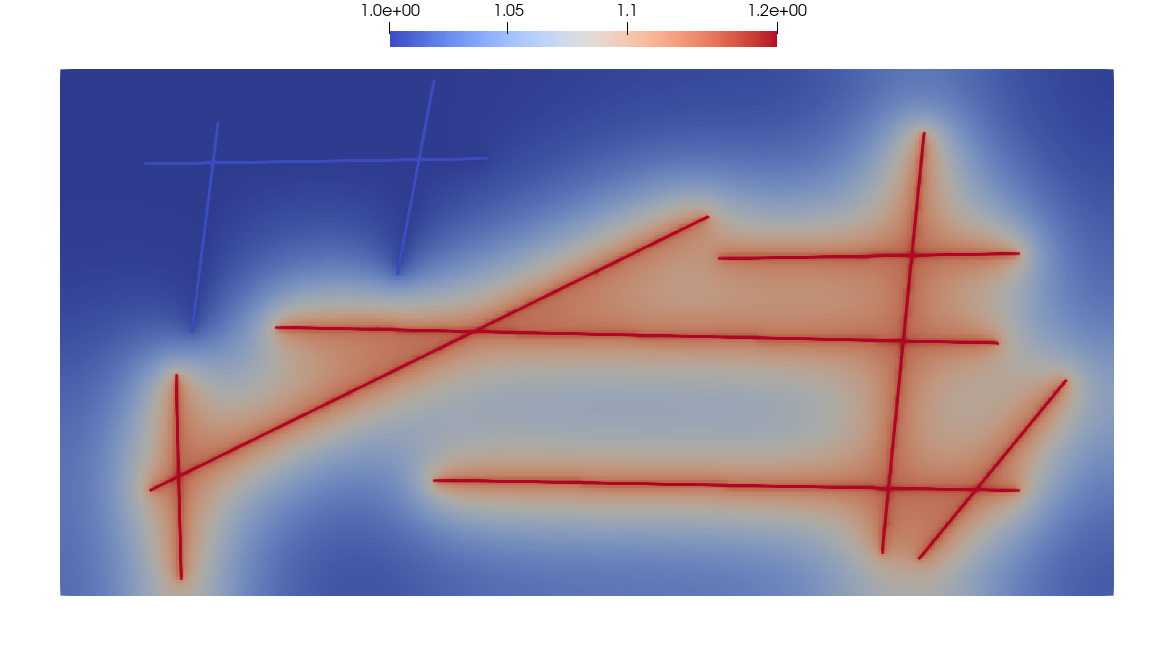}
\includegraphics[width=0.32 \textwidth]{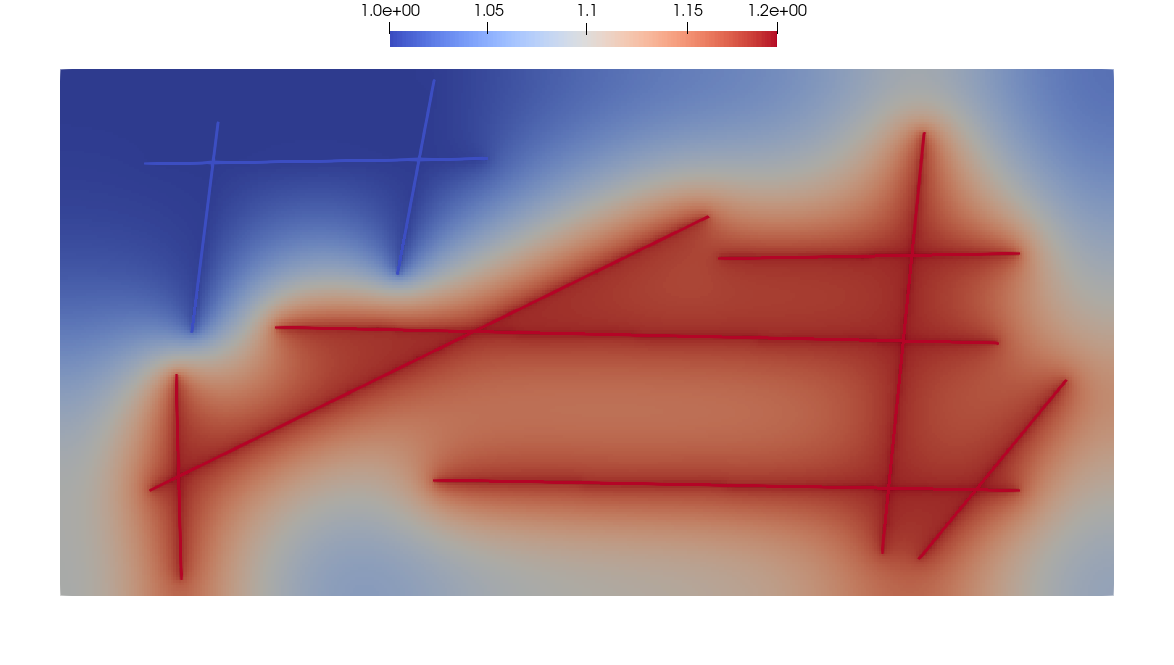}\\
\includegraphics[width=0.32 \textwidth]{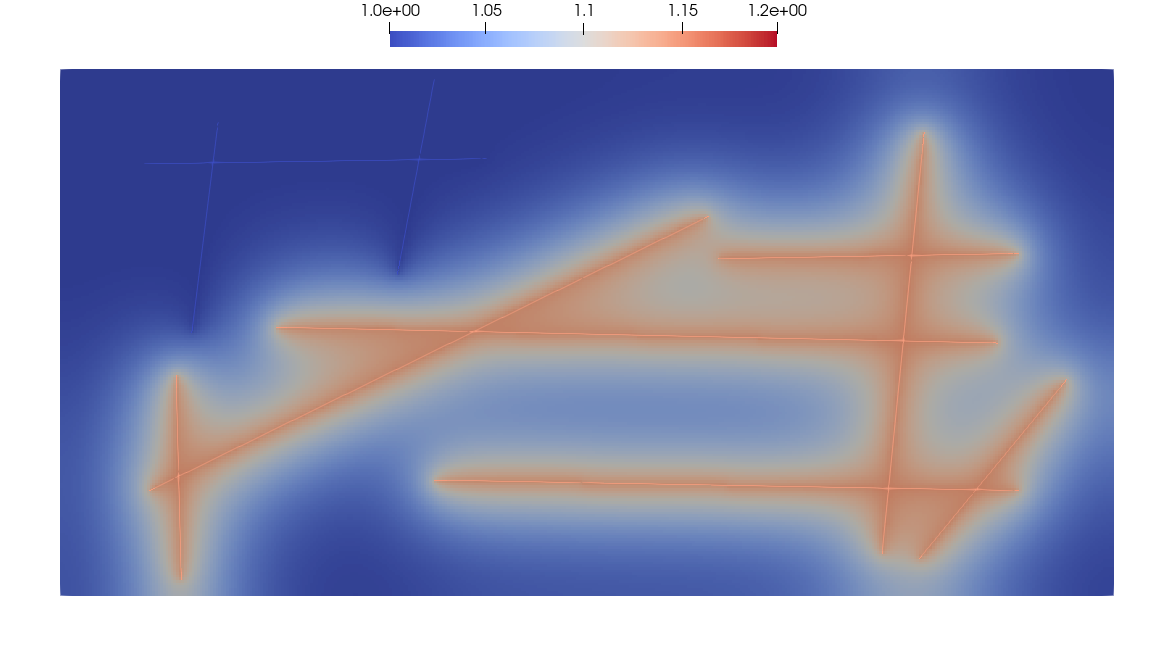}
\includegraphics[width=0.32 \textwidth]{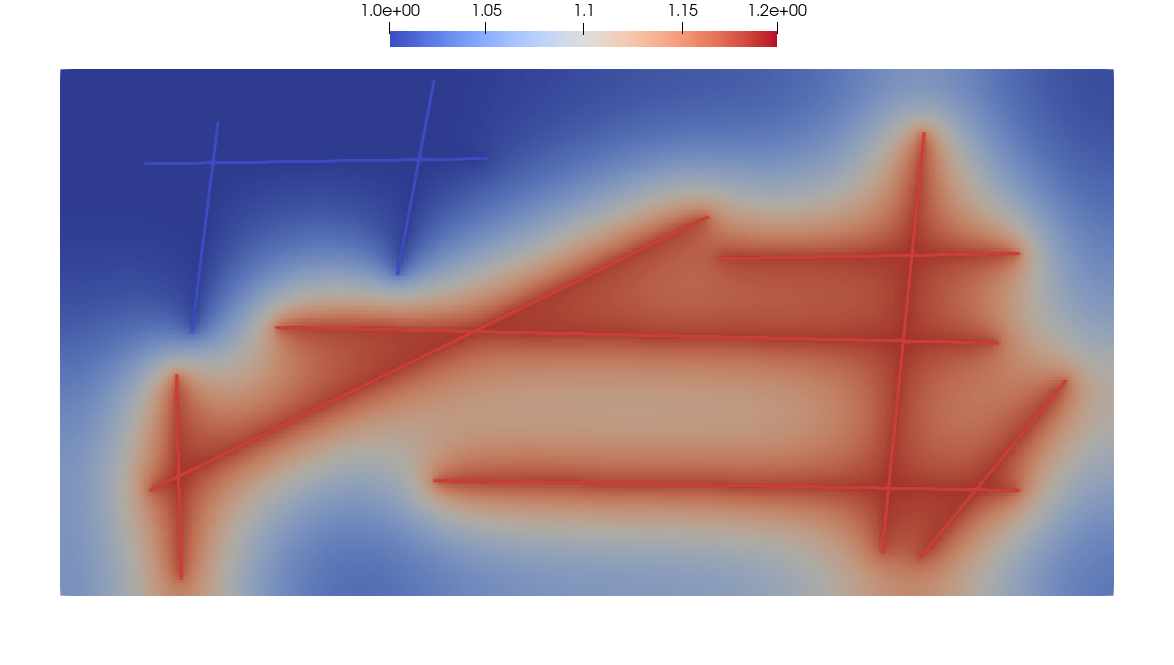}
\includegraphics[width=0.32 \textwidth]{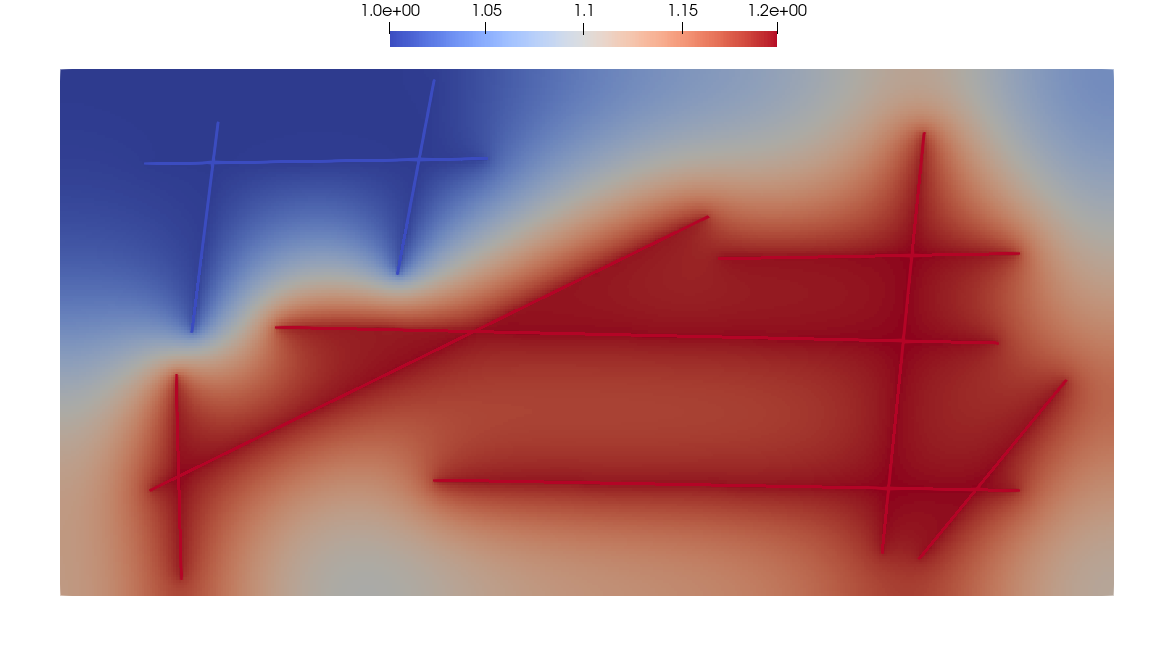}
\caption{Three-continuum media (\textit{3C}). Reference solution at $t = T_{max}/4, T_{max}/2$ and $T_{max}$ (from left to right). First row: fist continuum. Second row: second continuum}
\label{fig:u3-f}
\end{figure}

To perform a numerical investigation of the proposed Implicit and Implicit-Explicit schemes for finite-volume and multiscale approximations, we use the solution on the fine grid with $N_t = 1024$ using coupled scheme (implicit scheme \eqref{mm4} with $\theta=1$) as a reference solution. 
In Figure \ref{fig:u2-f} and \ref{fig:u3-f},  we present a solution for two- and three--continuum media, respectively. The solution $u^n$ is shown at three time layers $n=256, 512$ and $1024$ ($N_t = 1024$).  
In the coupled scheme, we solve a large coupled system of equations that have $DOF_h = 81,474$ for a two-continuum problem (\textit{2C}) and $DOF_h = 161,474$ for a three-continuum problem (\textit{3C}).  

Next, we numerically study schemes for various time step sizes $N_t = 4, 8, 16, 32$ and $64.128$ with $\tau = T_{max}/N_t$. We investigate the influence of the time step size on the method accuracy for coupled and decoupled schemes. 
We consider the following schemes:
\begin{itemize}
\item Coupled schemes:
\begin{itemize}
\item The two-level implicit scheme \eqref{mm4} with $\theta=1$ (Im1).
\item The three-level  implicit schemes \eqref{mm5} with $(\mu, \sigma) = (1,0)$ (Im2-CN,  Crank-Nicholson scheme) and $(\mu, \sigma) = (1.5,0)$ (Im2-BDF, second order backward differentiation formula).
\end{itemize}

\item Decoupled schemes:
\begin{itemize}
\item The two-level  implicit-explicit scheme \eqref{cm5} with $\theta=1$ (ImEx1).
\item The three-level implicit schemes \eqref{cm6} with $(\mu, \sigma) = (1,0)$ (ImEx2-CNAB,  Crank-Nicholson-Adams-Bashforth scheme) and $(\mu, \sigma) = (1.5,0)$ (ImEx2-SBDF, second order semi-implicit backward differentiation formula).
\end{itemize}
\end{itemize}
We also consider three variations of decoupled schemes: L, D, and U-scheme. The schemes are based on the different additive representations of the operator $A$ (see \eqref{cmL}, \eqref{cmD} and \eqref{cmU}) and related to the order of computations.  
Implementation is performed using the PETSc library \cite{balay2019petsc}, using a conjugate gradient (CG) iterative solver with ILU preconditioner. Simulations are performed on MacBook Pro (2.3 GHz Quad-Core Intel Core i7 with 32 GB 3733 MHz LPDDR4X).

\subsection{Implicit-Explicit for fine-scale system}

In this section, we consider a fine grid problem, where a finite-volume approximation is used to construct space approximation. We numerically investigate coupled and decoupled schemes. 

\begin{table}[h!]
\centering
\begin{tabular}{|c|cc|}
\hline
 \multicolumn{3}{|c|}{Implicit} \\ 
\hline
\multirow{ 2}{*}{$N_t$}
&   &\\
& $e_{h,1}$ (\%) & $e_{h,2}$ (\%)\\
\hline
\multicolumn{3}{|c|}{Im1} \\
\hline
4 	& 1.0059 & 17.9344 \\
8 	& 0.5217 & 9.7381 \\
16 	& 0.2673 & 5.0979 \\
32 	& 0.1375 & 2.6395 \\
64 	& 0.0719 & 1.3776 \\
128 & 0.0391 & 0.7395 \\
\hline
\multicolumn{3}{|c|}{Im2-CN} \\
\hline
4 	& 0.0893 & 2.0623 \\
8 	& 0.0247 & 0.8536 \\
16 	& 0.0058 & 0.7281 \\
32 	& 0.0013 & 0.6634 \\
64	& 0.0017 & 0.5349 \\
128 & 0.0021 & 0.3577 \\
\hline
\multicolumn{3}{|c|}{Im2-BDF} \\
\hline
4 	& 0.5519 & 11.9551 \\
8 	& 0.1397 & 3.0781 \\
16 	& 0.0320 & 0.6910 \\
32	& 0.0080 & 0.1791 \\
64 	& 0.0022 & 0.0589 \\
128 & 0.0009 & 0.0308 \\
\hline
\end{tabular}
\ \ \ \
\begin{tabular}{|c|cc|cc|cc|}
\hline
 \multicolumn{7}{|c|}{Implicit-Explicit} \\ 
\hline
\multirow{ 2}{*}{$N_t$}
& \multicolumn{2}{|c|}{L} 
& \multicolumn{2}{|c|}{D}
& \multicolumn{2}{|c|}{U}\\
& $e_{h,1}$ (\%) & $e_{h,2}$ (\%) 
& $e_{h,1}$ (\%) & $e_{h,2}$ (\%) 
& $e_{h,1}$ (\%) & $e_{h,2}$ (\%)   \\
\hline
 \multicolumn{7}{|c|}{ImEx1} \\ 
\hline 
4 
& 3.0290 & $>$100 
& 3.4885 & $>$100 
& 1.3912 & 21.9411 \\
8 
& 1.3841 & 34.6548
& 1.6220 & 39.9831
& 0.7023 & 11.9741 \\
16 
& 0.6551 & 14.2659
& 0.7578 & 16.0401
& 0.3522 & 6.2784 \\
32 
& 0.3201 & 6.6071
& 0.3655 & 7.3213
& 0.1782 & 3.2420 \\
64
& 0.1605 & 3.2350
& 0.1816 & 3.5600
& 0.0919 & 1.6813 \\
128 
& 0.0828 & 1.6410
& 0.0929 & 1.7968
& 0.0489 & 0.8920 \\
\hline
 \multicolumn{7}{|c|}{ImEx2-CNAB} \\ 
\hline 
4 
& 4.5677 & $>$100 
& 6.0503 & $>$100 
& 3.7071 & $>$100 \\
8 
& 0.4268 & 81.8828
& 5.4027 & $>$100 
& 0.3288 & 71.7085 \\
16 
& 0.1609 & 42.4335
& 1.8389 & $>$100 
& 0.1331 & 45.1967 \\
32 
& 0.0631 & 15.446
& 0.2309 & $>$100 
& 0.0722 & 20.8431 \\
64 
& 0.0163 & 1.7747
& 0.0337 & 8.4869
& 0.0496 & 4.1516 \\
128
& 0.0086 & 0.3953
& 0.0010 & 0.3870
& 0.0402 & 0.9422 \\
\hline
 \multicolumn{7}{|c|}{ImEx2-SBDF} \\ 
\hline 
4 
& 2.8739 & $>$100 
& 3.0840 & 91.0879
& 2.5372 & $>$100  \\
8 
& 0.6759 & 25.6489
& 0.8206 & 32.7378
& 0.3708 & 9.8857 \\
16 
& 0.2517 & 5.3571
& 0.3213 & 6.4609
& 0.0721 & 3.3658 \\
32 
& 0.1042 & 2.1288
& 0.1389 & 2.6478
& 0.0342 & 1.6021 \\
64 
& 0.0280 & 0.5888 
& 0.0449 & 0.8388
& 0.0372 & 1.0167 \\
128 
& 0.0121 & 0.2169
& 0.0029 & 0.0729
& 0.0432 & 0.9162 \\
\hline
\end{tabular}
\caption{Two-continuum media (\textit{2C}). Relative errors in \% for coupled (Implicit, Im) and decoupled (Implicit-Explicit, ImEx) schemes}
\label{tab2-f}
\end{table}

\begin{table}[h!]
\centering
\begin{tabular}{|c|cc|}
\hline
 \multicolumn{3}{|c|}{Implicit} \\ 
\hline
\multirow{ 2}{*}{$N_t$}
&   &\\
& $e_{h,1}$ (\%) & $e_{h,2}$ (\%)\\
\hline
\multicolumn{3}{|c|}{Im1} \\
\hline
4 & 1.0600 & 16.3468 \\
8 & 0.5687 & 9.2291  \\
16 & 0.2974 & 4.9515  \\
32 & 0.1548 & 2.5998  \\
64 & 0.0818 & 1.3684  \\
128 & 0.0449 & 0.7396  \\
\hline
\multicolumn{3}{|c|}{Im2-CN} \\
\hline
4 & 0.0918 & 1.7399  \\
8 & 0.0253 & 0.7161  \\
16 & 0.0063 & 0.6049  \\
32 & 0.0024 & 0.5508  \\
64 & 0.0027 & 0.4443  \\
128 & 0.0030 & 0.2978  \\
\hline
\multicolumn{3}{|c|}{Im2-BDF} \\
\hline
4 & 0.4052 & 10.0464  \\
8 & 0.1216 & 2.2998 \\
16 & 0.0313 & 0.5631  \\
32 & 0.0084 & 0.1577  \\
64 & 0.0026 & 0.0580  \\
128 & 0.0012 & 0.0332  \\
\hline
\end{tabular}
\ \ \ \
\begin{tabular}{|c|cc|cc|cc|}
\hline
 \multicolumn{7}{|c|}{Implicit-Explicit} \\ 
\hline
\multirow{ 2}{*}{$N_t$}
& \multicolumn{2}{|c|}{L} 
& \multicolumn{2}{|c|}{D}
& \multicolumn{2}{|c|}{U}\\
& $e_{h,1}$ (\%) & $e_{h,2}$ (\%) 
& $e_{h,1}$ (\%) & $e_{h,2}$ (\%) 
& $e_{h,1}$ (\%) & $e_{h,2}$ (\%)   \\
\hline
 \multicolumn{7}{|c|}{ImEx1} \\ 
\hline 
4 
& 5.1618 & $>$100 
& 5.6480 & $>$100 
& 2.0955 & 33.2298\\
8
& 2.5859 & 47.0465
& 3.0596 & 55.8464
& 1.1479 & 19.1050\\
16
& 1.2739 & 21.7915
& 1.5668 & 26.6997
& 0.6052 & 10.2833\\
32
& 0.6324 & 10.5583
& 0.7912 & 13.1384
& 0.3139 & 5.3677\\
64
& 0.3179 & 5.2425
& 0.4001 & 6.5626
& 0.1630 & 2.7810\\
128
& 0.1625 & 2.6555
& 0.2044 & 3.3234
& 0.0862, & 1.4560\\
\hline
 \multicolumn{7}{|c|}{ImEx2-CNAB} \\ 
\hline 4 
& 6.4256 & $>$100 
& 6.6215 & $>$100 
& 4.5543 & $>$100 \\
8
& 0.3956 & 70.6279
& 4.3660 & $>$100 
& 1.3503 & 65.1615 \\
16
& 0.1244 & 35.4564
& 1.4028 & $>$100 
& 0.6573 & 39.1486 \\
32
& 0.0639 & 12.8874
& 0.1987 & $>$100 
& 0.3406 & 18.1321 \\
64
& 0.0473 & 1.8030
& 0.0437 & 7.0812
& 0.1888 & 4.5650 \\
128
& 0.0428 & 0.8537
& 0.0018 & 0.3221
& 0.1145 & 1.9190 \\
\hline
 \multicolumn{7}{|c|}{ImEx2-SBDF} \\ 
\hline 
4 
& 4.3465 & $>$100 
& 4.7712 & 98.5998
& 3.7142 & 223.4047 \\
8
& 0.6878 & 25.6575
& 1.1494 & 30.5044
& 1.3159 & 22.7793 \\
16
& 0.1249 & 4.2209
& 0.4291 & 6.7873
& 0.5516 & 9.7356 \\
32
& 0.0417 & 1.6920
& 0.1823 & 2.8165
& 0.2875 & 5.0038 \\
64
& 0.0336 & 0.9427
& 0.0590 & 0.8991
& 0.1723 & 2.9180 \\
128
& 0.0473 & 0.8639
& 0.0040 & 0.0873
& 0.1182 & 1.9514 \\
\hline
\end{tabular}
\caption{Three-continuum media (\textit{3C}). Relative errors in \% for coupled (Implicit, Im) and decoupled (Implicit-Explicit, ImEx) schemes}
\label{tab3-f}
\end{table}

To compare schemes, we calculate $L_2$ and energy relative errors in percentage  at time $t$ 
\[
e_{h,1}(t) = \frac{||u_h(t) - \tilde{u}_h(t)||}{||u_h(t)||} \times 100 \%, 
\quad
e_{h,2}(t) = \frac{||u_h(t) - \tilde{u}_h(t)||_{A}}{||u_h(t)||_{A}} \times 100 \%, 
\]
with 
\[
||u|| =  \sqrt{(u, u)}, \quad 
||u||_A =  \sqrt{(Au, u)} 
\]
where $u_h$ is the reference solution (coupled Im1-scheme with $N_t=1024$), and $\tilde{u}$ is the approximate solution using different schemes (Im1, Im2-CN, Im2-BDF, ImEx1, ImEx2-CNAB and ImEx2-SBDF) and $N_t = 4,8,16,32,64$ and $128$.

\begin{figure}[h!]
\centering
\begin{subfigure}[b]{0.45\textwidth}
\centering
\includegraphics[width=1 \textwidth]{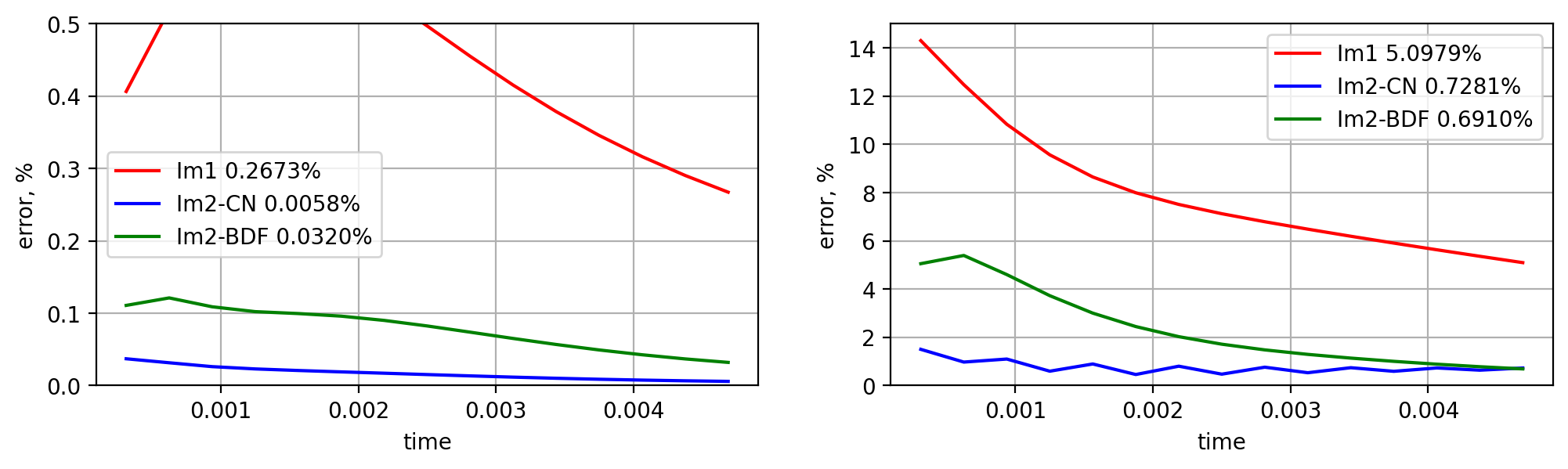}
\includegraphics[width=1 \textwidth]{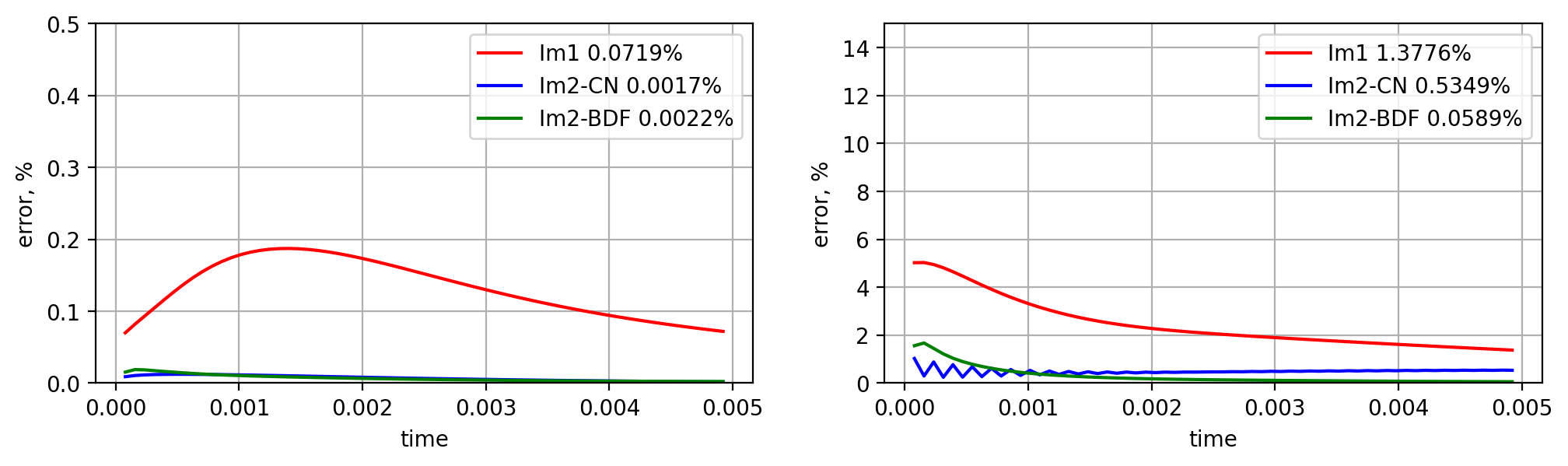}
\caption{Two-continuum media (\textit{2C}). Errors $e_{h,1}(t)$ (left) and $e_{h,2}(t)$ (right)}
\end{subfigure}
\ \ \ \ 
\begin{subfigure}[b]{0.45\textwidth}
\centering
\includegraphics[width=1 \textwidth]{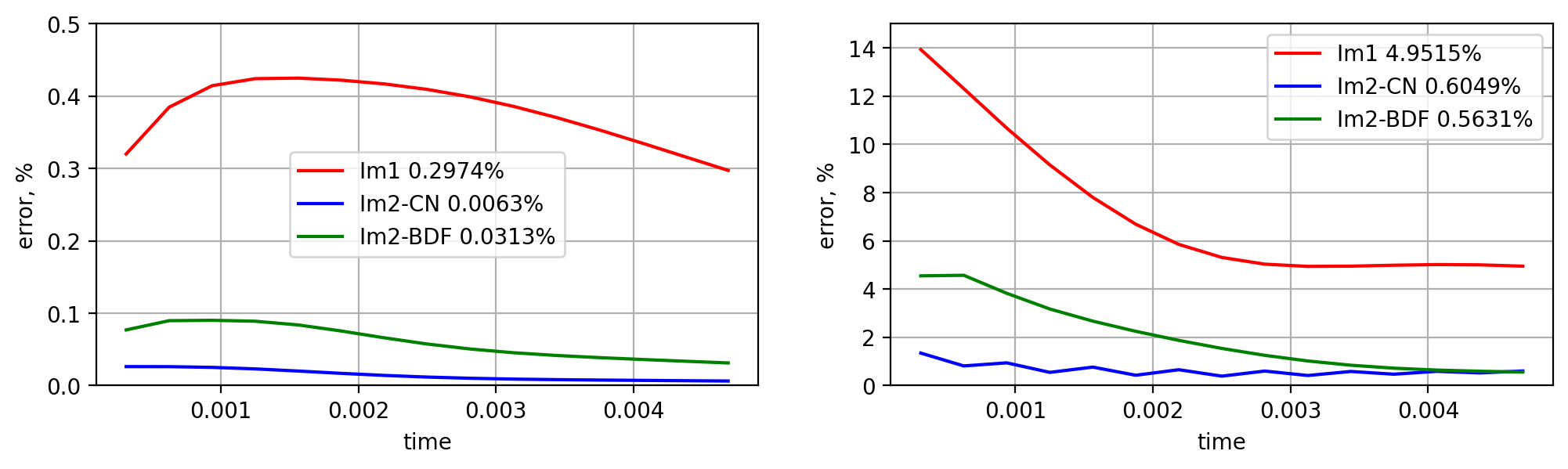}
\includegraphics[width=1 \textwidth]{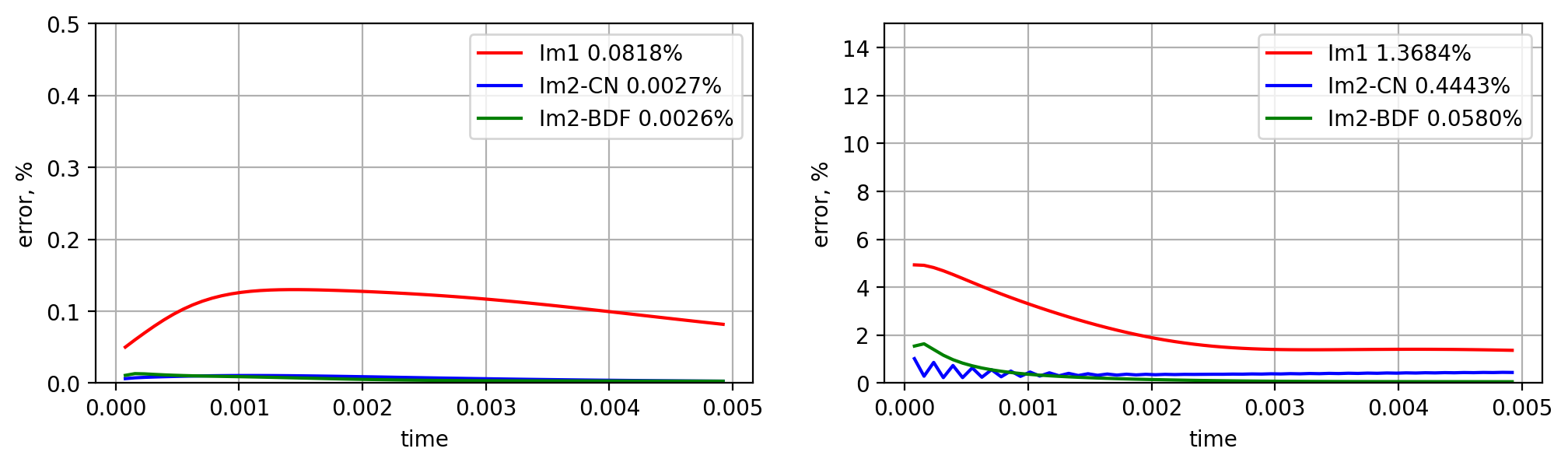}
\caption{Three-continuum media (\textit{3C}). Errors $e_{h,1}(t)$ (left) and $e_{h,2}(t)$ (right)}
\end{subfigure}
\caption{The dynamic of the error with $N_t = 16$ (first row) and $64$ (second row). The label is given with the error at the final time}
\label{fig:u-et}
\end{figure}

\begin{table}[h!]
\centering
\begin{tabular}{|c|cc|}
\hline
 \multicolumn{3}{|c|}{Implicit} \\ 
\hline
$N_t$ 
& time$_{tot}$ & $N_{AvIt}$ \\
\hline
\multicolumn{3}{|c|}{Im1} \\
\hline
4 	& 2.294 & 429.5 \\
8 	& 3.869 & 362.0 \\
16 	& 6.441 & 280.6 \\
32 	& 9.747 & 227.0 \\
64 	& 15.770 & 184.1 \\
128 & 25.311 & 146.5 \\
\hline
\multicolumn{3}{|c|}{Im2-CN} \\
\hline
4 	& 1.450 & 359.7  \\
8 	& 2.605 & 278.1 \\
16 	& 4.558 & 227.1 \\
32 	& 7.653 & 184.2 \\
64 	& 12.412 & 146.5 \\
128 & 20.322 & 117.9 \\
\hline
\multicolumn{3}{|c|}{Im2-BDF} \\
\hline
4 	& 1.582 & 391.3 \\
8 	& 2.985 & 316.1 \\
16 	& 4.948 & 244.2 \\
32 	& 8.404 & 201.6 \\
64 	& 13.563 & 160.4 \\
128 & 22.015 & 129.7 \\
\hline
\end{tabular}
\ \ \ \
\begin{tabular}{|c|c|cc|cc|}
\hline
 \multicolumn{6}{|c|}{Implicit-Explicit} \\ 
\hline
$N_t$ 
& time$_{tot}$ 
& time$_{tot}^m$ & $N_{AvIt}^m$ 
& time$_{tot}^f$ & $N_{AvIt}^f$   \\
\hline
 \multicolumn{6}{|c|}{ImEx1} \\ 
\hline 
4 	& 0.842 & 0.837 & 160.0 & 0.004  & 59.2 \\
8 	& 1.314 & 1.306 & 124.0 & 0.008  & 58.1 \\
16 	& 1.994 & 1.977 & 93.9 	& 0.016  & 57.8 \\
32 	& 2.907 & 2.874 & 68.0 	& 0.033  & 57.1 \\
64 	& 4.347 & 4.283 & 50.0 	& 0.064  & 55.5 \\
128 & 6.351 & 6.227 & 36.0 	& 0.124  & 53.3 \\
\hline
 \multicolumn{6}{|c|}{ImEx2-CNAB} \\ 
\hline 
4 	& 0.550 & 0.545  & 124.7& 0.004  & 59.0  \\
8 	& 1.018 & 1.009  & 93.9 & 0.009  & 57.3  \\
16	& 1.654 & 1.633  & 68.0 & 0.020  & 57.3  \\
32	& 2.704 & 2.662  & 50.0 & 0.042  & 55.5  \\
64 	& 4.284 & 4.203  & 36.0 & 0.081  & 53.1 \\
128	& 6.777 & 6.609  & 25.0 & 0.167  & 53.0 \\
\hline
 \multicolumn{6}{|c|}{ImEx2-SBDF} \\ 
\hline 
4 	& 0.602 & 0.598  & 139.0 	& 0.003  & 58.3   \\
8 	& 1.105 & 1.095  & 105.0 	& 0.009  & 57.1  \\
16 	& 1.828 & 1.808  & 78.0 	& 0.020  & 57.4  \\
32 	& 2.935 & 2.893  & 57.0 	& 0.042  & 56.8 \\
64 	& 4.651 & 4.568  & 41.0 	& 0.083  & 54.1 \\
128 & 7.517 & 7.347  & 29.0 	& 0.169  & 53.0   \\
\hline
\end{tabular}
\caption{Two-continuum media (\textit{2C}). Time of the solution and the average number of iterations.  Coupled (Implicit, Im) and decoupled (Implicit-Explicit, ImEx) with $N_h = 81474$. Reference solution: time$_{tot} = 103.76$ sec. and $N_{AverIt} = 73.98$ for $N_t = 1024$ (Im1)}
\label{tab2-f-t}
\end{table}

\begin{table}[h!]
\centering
\begin{tabular}{|c|cc|}
\hline
 \multicolumn{3}{|c|}{Implicit} \\ 
\hline
$N_t$ 
& time$_{tot}$ & ${N}_{AvIt}$ \\
\hline
\multicolumn{3}{|c|}{Im1} \\
\hline
4 	& 4.527 & 369.8  \\
8 	& 7.698 & 315.9  \\
16 	& 12.460 & 256.3  \\
32 	& 20.605 & 213.4  \\
64 	& 33.814 & 175.0  \\
128 & 53.684 & 140.1  \\
\hline
\multicolumn{3}{|c|}{Im2-CN} \\
\hline
4 	& 2.833 & 314.7  \\
8 	& 5.452 & 256.1  \\
16 	& 9.701 & 213.3  \\
32 	& 16.706 & 175.0  \\
64 	& 27.507 & 140.1  \\
128 & 46.016 & 115.4  \\
\hline
\multicolumn{3}{|c|}{Im2-BDF} \\
\hline
4 	& 3.069 & 336.0  \\
8 	& 5.976 & 280.4  \\
16 	& 10.483 & 229.3  \\
32 	& 18.220 & 191.9  \\
64 	& 29.982 & 154.6  \\
128 & 49.034 & 125.1  \\
\hline
\end{tabular}
\ \ \ \
\begin{tabular}{|c|c|cc|cc|cc|}
\hline
 \multicolumn{8}{|c|}{Implicit-Explicit} \\ 
\hline
$N_t$ 
& time$_{tot}$ 
& time$_{tot}^1$ & ${N}_{AvIt}^1$ 
& time$_{tot}^2$ & ${N}_{AvIt}^2$  
& time$_{tot}^f$ & ${N}_{AvIt}^f$   \\
\hline
 \multicolumn{8}{|c|}{ImEx1} \\ 
\hline 
4 	& 0.767 & 0.042 & 7.0 & 0.720 & 139.0 & 0.004 & 59.2 \\
8 	& 1.267 & 0.074 & 6.0 & 1.185 & 113.6 & 0.008 & 59.0 \\
16 	& 2.012 & 0.130 & 5.0 & 1.865 & 89.0 & 0.016 & 58.3 \\
32 	& 3.029 & 0.213 & 4.0 & 2.783 & 66.0 & 0.032 & 57.0 \\
64 	& 4.645 & 0.429 & 4.0 & 4.151 & 49.0 & 0.064 & 55.9 \\
128 & 6.858 & 0.698 & 3.0 & 6.037 & 35.0 & 0.122 & 53.2 \\
\hline
 \multicolumn{8}{|c|}{ImEx2-CNAB} \\ 
\hline 
4 	& 0.584 & 0.080 & 6.0 & 0.499 & 113.7 & 0.004 & 59.3 \\
8 	& 1.128 & 0.179 & 5.0 & 0.938 & 88.9 & 0.009 & 57.6 \\
16 	& 1.953 & 0.370 & 4.0 & 1.563 & 66.0 & 0.020 & 57.0 \\
32 	& 3.358 & 0.758 & 4.0 & 2.558 & 49.0 & 0.041 & 55.7 \\
64 	& 5.704 & 1.525 & 3.0 & 4.097 & 35.0 & 0.081 & 53.2 \\
128 & 9.601 & 2.928 & 3.0 & 6.512 & 25.0 & 0.159 & 52.9 \\
\hline
 \multicolumn{8}{|c|}{ImEx2-SBDF} \\ 
\hline 
4 	& 0.625 & 0.085 & 7.0 & 0.536 & 124.0 & 0.004 & 59.0 \\
8 	& 1.217 & 0.189 & 6.0 & 1.019 & 98.9 & 0.009 & 58.3 \\
16	& 2.139 & 0.394 & 5.0 & 1.724 & 75.0 & 0.019 & 57.1 \\
32 	& 3.613 & 0.749 & 4.0 & 2.822 & 56.0 & 0.041 & 56.5 \\
64 	& 6.126 & 1.534 & 3.0 & 4.509 & 40.0 & 0.081 & 54.3 \\
128 & 10.389 & 3.016 & 3.0 & 7.211 & 29.0 & 0.161 & 53.0 \\
\hline
\end{tabular}
\caption{Three-continuum media (\textit{3C}). Time of the solution and the average number of iterations.  Coupled (Implicit, Im) and decoupled (Implicit-Explicit, ImEx)}
\label{tab3-f-t}
\end{table}

In Tables \ref{tab2-f} and \ref{tab3-f}, we present errors for two and three-continuum media, respectively. We show an error in percentage at the final time. We observe the following convergence behavior of the proposed schemes from the result. 
For the three-level  schemes, we have a smaller error. 
For example, in Table \ref{tab2-f} for a two-continuum test problem and in Table \ref{tab3-f} for a three-continuum test problem, we have less than 1\% of energy error for $N_t = 128$ time steps in the two-level scheme, Im1. For three-level  implicit schemes, we have less than 1\%  energy error for $N_t = 8$ and $N_t = 16$ time steps in Im2-CN and Im2-BDF, respectively. 
In the uncoupled case (D-schemes) for two-continuum case in Table \ref{tab2-f}, we have near 2\% of energy error in two-level  decoupled scheme (ImEx1) for $N_t = 128$ and $N_t = 32$ for three-level  decoupled scheme (ImEx2-SBDF). 
For the three-continuum case in Table \ref{tab3-f}, we have nearly 3\% of energy error in the first-order decoupled scheme (ImEx1) for $N_t = 128$ and $N_t = 32$ for second-order decoupled scheme (ImEx2-SBDF). 
Therefore, we can use a 4-times bigger time step size in the second-order schemes.

Next, we compare the coupled and decoupled schemes.
For the two-level  schemes in Table \ref{tab2-f}, we have nearly 1\%  of energy error for $N_t = 64$ time steps in both coupled and decoupled schemes (Im1 and ImEx1-U). Moreover, we observe that the U-decoupling scheme gives better results compared with L and D-schemes in ImEx1. We see that continuum order matters in calculations: (1) It is better to use previous continuum solution information in calculations, and (2) First, calculate the continuum with higher permeability or faster flow. 
For the second-order decoupled scheme, we observe that the errors are smaller for the L-scheme in ImEx2-CNAB and ImEx2-SBDF. However, it is better to use a U-scheme for a smaller time step size. Moreover, we also observe that the D-scheme gives more considerable errors for ImEx2-CNAB with fewer time steps and can lead to a big energy error.
By comparing the coupled and uncoupled schemes of the three-level, we observe a more significant influence on the method accuracy. For example, we can use $N_t = 16$ in the Im2-BDF coupled scheme to obtain near 1\%  of energy error, but in the ImEx2-SBDF decoupled scheme, we should use $N_t = 32$.

In Figure \ref{fig:u-et}, we depicted the dynamics of the error by time. We show the errors in $L_2$ and energy norms for $N_t = 16$ on the first row and for $N_t = 64$ on the second row. The figures illustrate a good performance of the second-order schemes compared with a first-order for a coupled case of calculations. We observe a well-known zigzag behavior of the Crank-Nicholson scheme at the beginning of simulations. When we use a more significant time step size ($N_t=16$, first row in Figure \ref{fig:u-et}), we observe more minor errors for the Im2-CN scheme compared with the Im2-BDF for both two-continuum and three-continuum test problems. However, for $N_t=64$, we follow a more minor error for the Im2-BDF scheme. 
Finally, we discuss the computational performance of the presented schemes for the fine grid problem. We show the total time of computations in seconds with an average number of iterations for the preconditioned conjugate gradient method in Tables \ref{tab2-f-t} and \ref{tab3-f-t}.  In both tables, time$_{tot}$ denotes the total time of computations for coupled schemes for given $N_t$ and time$_{tot}^{\alpha}$  is the total time of calculations related to the $\alpha$-continuum in decoupled schemes, $\alpha=m, f$ for two-continuum test and $\alpha=1,2,f$ for three-continuum test. ${N}_{AvIt}$ denotes an average number of iterations in one time layer, and ${N}_{AvIt}^{\alpha}$ is related to the $\alpha$-continuum in the decoupled schemes.  
For the decoupled schemes, we have the same computational time and number of iterations in L, D, and U-schemes. Therefore, the results in Tables \ref{tab2-f-t} and \ref{tab3-f-t} are presented for U-scheme. 
To compare the computational efficiency, we present the reference solution's time and number of iterations.
\begin{itemize}
\item 
Two-continuum media (\textit{2C}): $N_h = 81474$, time$_{tot} = 103.76$ sec. and $N_{AverIt} = 73.98$  using coupled (Im1) scheme for $N_t = 1024$.
\item 
Three-continuum media (\textit{3C}): $N_h = 161474$, time$_{tot} = 234.37$ sec. and $N_{AverIt} = 73.90$  using coupled (Im1) scheme for $N_t = 1024$.
\end{itemize}
We observe the same computation time for first- and second-order schemes from the presented results. 
For the coupled schemes for the two-continuum test, we have $12-15$ sec for $N_t=64$ and $20-25$ sec for $N_t=128$. 
In the coupled case on each time iteration, we solve the linear system with $N_h = 81474$ and $N_h=161474$ unknowns for two and three-continuum tests, respectively. In the decoupled case, we decoupled the system into systems with smaller sizes that are faster to solve. We have $3-4$ times faster calculations for the two-continuum test with nearly $4$ sec for $N_t=64$ and $6-7$ sec for $N_t=128$. In the three-continuum cases, we obtain better results with $46-53$ sec for the coupled schemes and $6-10$ sec for decoupled implicit-explicit schemes for $N_t=128$. Then, we have $7.8$ times faster calculations for the two-level  scheme and $4.8$ times faster calculations for the second order scheme. The computational time is reduced due to the small size of the system in the decoupled case that requires a smaller number of iterations. The implicit-explicit schemes also decouple continuums and significantly affect the number of iterations. Furthermore, we observe that the most time is taken by a first continuum defined in $\Omega$ in \textit{2C} test and by second continuum $p_2 \in \Omega$ in \textit{3C} test. The fracture continuum requires a more significant number of iterations in both test cases. However, the time of calculations is very short due to the size of the fracture continuum system $N_h^f = 1474$ compared with the continuum defined in $\Omega$ with $N_h^m = 80,000$ for \textit{2C} and $N_h^1=N_h^2 = 80,000$ for \textit{3C}.

\subsection{Implicit-Explicit for coarse-scale system}

This section considers a coarse grid problem constructed with NLMC approximation. 
Similar to the previous subsection, we numerically investigate coupled and decoupled schemes. In this case, we are more interested in the coarse grid system errors; however, we can reconstruct the fine-scale solution using the constructed projection operator. 
We use a $L_2$ relative error in percentage at time $t$  on fine and coarse grids to compare the proposed schemes.
\[
e_{ms,1}(t) = \frac{||u_h(t) - u_{ms}(t)||}{||u_h(t)||} \times 100 \%, 
\quad
e_{H,1}(t) = \frac{||\bar{u}_h(t) - u_H(t)||}{||\bar{u}_h(t)||} \times 100 \%, 
\]
with 
\[
\bar{u}_h(t) = \frac{1}{|K_i|} \int_{K_i} u_h(t) dx
\]
where $u_h$ is the reference solution,  $\bar{u}_h$ is the coarse grid reference solution (average over coarse cell $K_i$), $u_{ms}$ is the multiscale solution on the fine grid, and $u_H$ is the multiscale solution on the coarse grid.

\begin{table}[h!]
\centering
\begin{tabular}{|c|cc|cc|cc|cc|cc|}
\hline
 \multicolumn{11}{|c|}{Implicit schemes + Ms} \\ 
\hline
\multirow{ 2}{*}{$N_t$}
& \multicolumn{2}{|c|}{$K_i^{+,1}$} 
&  \multicolumn{2}{|c|}{$K_i^{+,2}$} 
&  \multicolumn{2}{|c|}{$K_i^{+,3}$} 
&  \multicolumn{2}{|c|}{$K_i^{+,4}$} 
&  \multicolumn{2}{|c|}{$K_i^{+,5}$} \\
& $e_{ms,1}$ & $e_{H,1}$
& $e_{ms,1}$ & $e_{H,1}$
& $e_{ms1}$  & $e_{H,1}$ 
& $e_{ms,1}$ & $e_{H,1}$ 
& $e_{ms,1}$ & $e_{H,1}$ \\
\hline
\multicolumn{11}{|c|}{Ms-Im1} \\
\hline	
4 
& 83.9536 & 86.4985 
& 3.2716 & 1.0142 & 1.2637 & 0.9429 & 1.0209 & 0.9389 & 1.0054 & 0.9390	\\
8 
& 85.4075 & 87.7092 
& 3.1621 & 0.5784 & 0.9254 & 0.4864 & 0.5502 & 0.4812 & 0.5201 & 0.4813	\\
16 
& 85.8038 & 88.0360 
& 3.1315 & 0.3608 & 0.8098 & 0.2491 & 0.3218 & 0.2431 & 0.2664 & 0.2431	\\
32 
& 85.9576 & 88.1630 
& 3.1232 & 0.2584 & 0.7769 & 0.1289 & 0.2286 & 0.1221 & 0.1397 & 0.1221	\\
64 
& 85.9905 & 88.1897	
& 3.1209 & 0.2122 & 0.7681 & 0.0691 & 0.1982 & 0.0613 & 0.0805 & 0.0613	\\
128 
& 85.8956 & 88.1091 
& 3.1202 & 0.1916 & 0.7658 & 0.0401 & 0.1899 & 0.0311 & 0.0566 & 0.0310	\\
\hline
\multicolumn{11}{|c|}{Ms-Im2-CN} \\
\hline	
4 
& 81.6139 & 84.5647 
& 3.6606 & 1.9931 & 1.9531 & 1.9128 & 1.7926 & 1.9046 & 1.7834 & 1.9048	\\
8 
& 93.0646 & 94.1468 
& 3.4462 & 1.4434 & 1.5967 & 1.4458 & 1.4083 & 1.4442 & 1.3965 & 1.4446	\\
16 
& 99.4061 & 99.4906 
& 3.4133 & 1.3131 & 1.5556 & 1.3494 & 1.3653 & 1.3508 & 1.3532 & 1.3513	\\
32 
& 100.006 & 100.012
& 3.4074 & 1.2780 & 1.5548 & 1.3282 & 1.3659 & 1.3308 & 1.3538 & 1.3313	\\
64 
& 99.9998 & 99.9984 
& 3.4067 & 1.2671 & 1.5586 & 1.3234 & 1.3709 & 1.3265 & 1.3589 & 1.3270	\\
128 
& 100.000 & 100.001 
& 3.4068 & 1.2631 & 1.5615 & 1.3223  & 1.3745 & 1.3256 & 1.3625 & 1.3261	\\
\hline
\multicolumn{11}{|c|}{Ms-Im2-BDF} \\
\hline	
4 
& 81.9196 & 84.7940 
& 5.1153 & 3.8363 & 4.1302 & 3.7634 & 4.0598 & 3.7587 & 4.0565 & 3.7589	\\
8 
& 84.5305 & 86.9679 
& 3.4730 & 1.4733 & 1.6554 & 1.3504 & 1.4728 & 1.3418 & 1.4618 & 1.3418	\\
16 
& 85.3065 & 87.6165 
& 3.1968 & 0.7177  & 0.9912 & 0.5817 & 0.6502 & 0.5723 & 0.6244 & 0.5722	\\
32 
& 85.6212 & 87.8796 
& 3.1411 & 0.4220 & 0.8223 & 0.2801 & 0.3486 & 0.2706 & 0.2977 & 0.2704	\\
64 
& 85.6438 & 87.8975 
& 3.1268 & 0.2907 & 0.7799 & 0.1431 & 0.2365 & 0.1334 & 0.1519 & 0.1332	\\
128 
& 85.3894 & 87.6826 
& 3.1225 & 0.2302 & 0.7691 & 0.0779 & 0.2009 & 0.0678 & 0.0865 & 0.0676	\\
\hline
\end{tabular}
\caption{Two-continuum media (\textit{2C}). Relative errors in \% of coupled schemes for multiscale approximation}
\label{tab2-ms-ov}
\end{table}

\begin{figure}[h!]
\centering
\includegraphics[width=0.19 \textwidth]{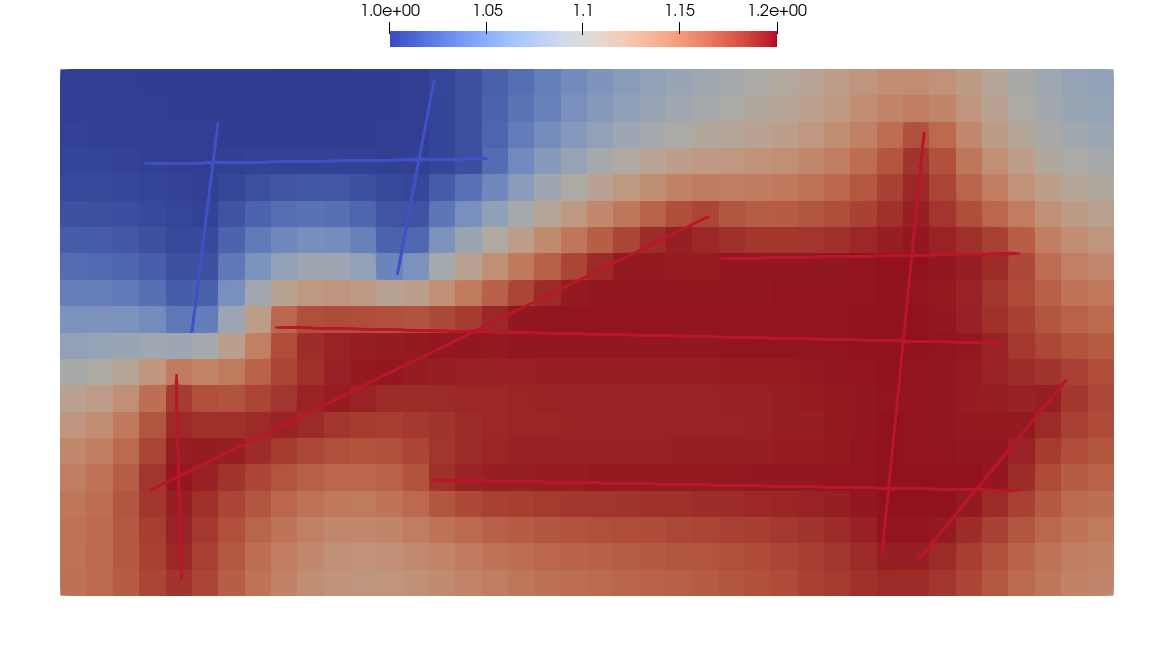}
\includegraphics[width=0.19 \textwidth]{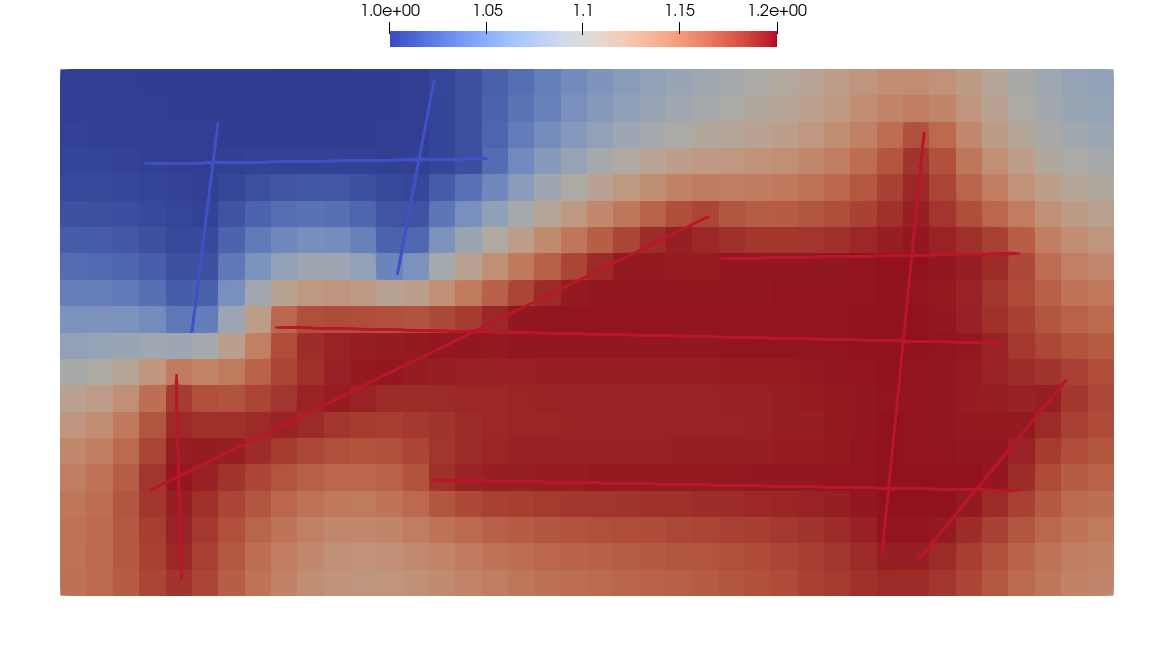}
\includegraphics[width=0.19 \textwidth]{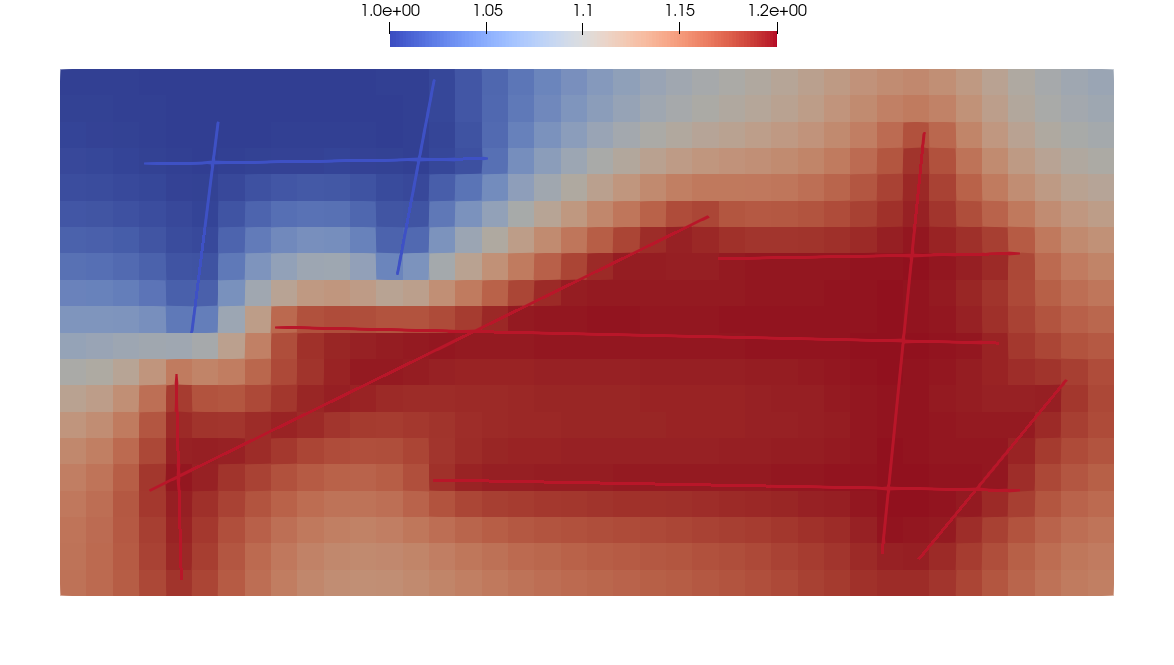}
\includegraphics[width=0.19 \textwidth]{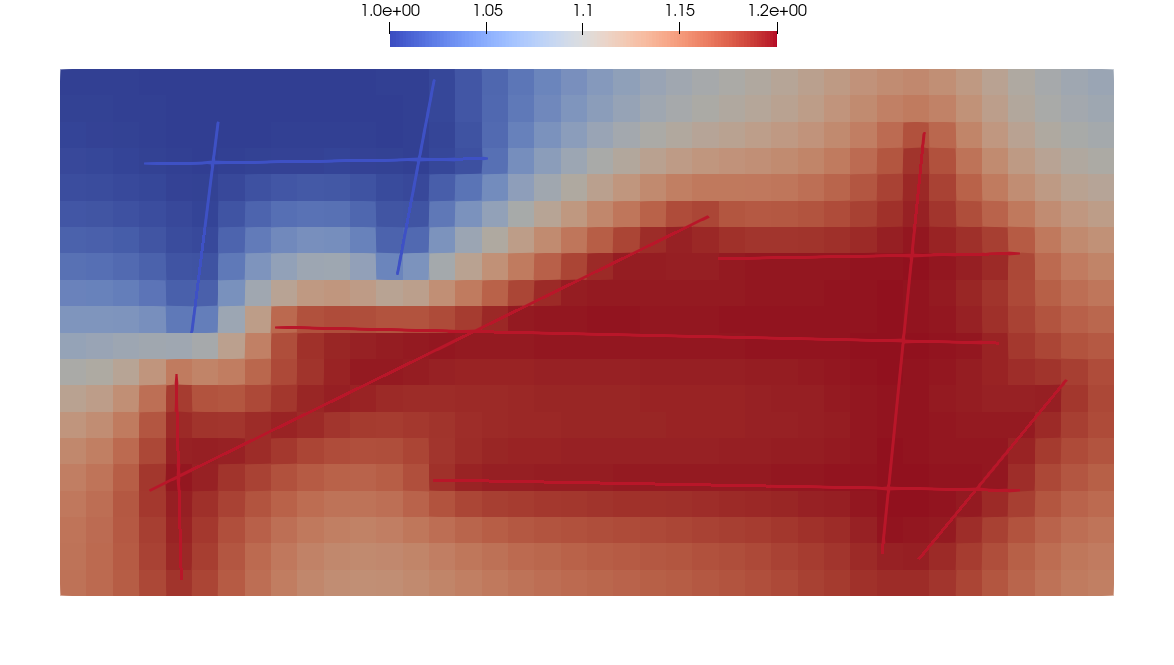}
\includegraphics[width=0.19 \textwidth]{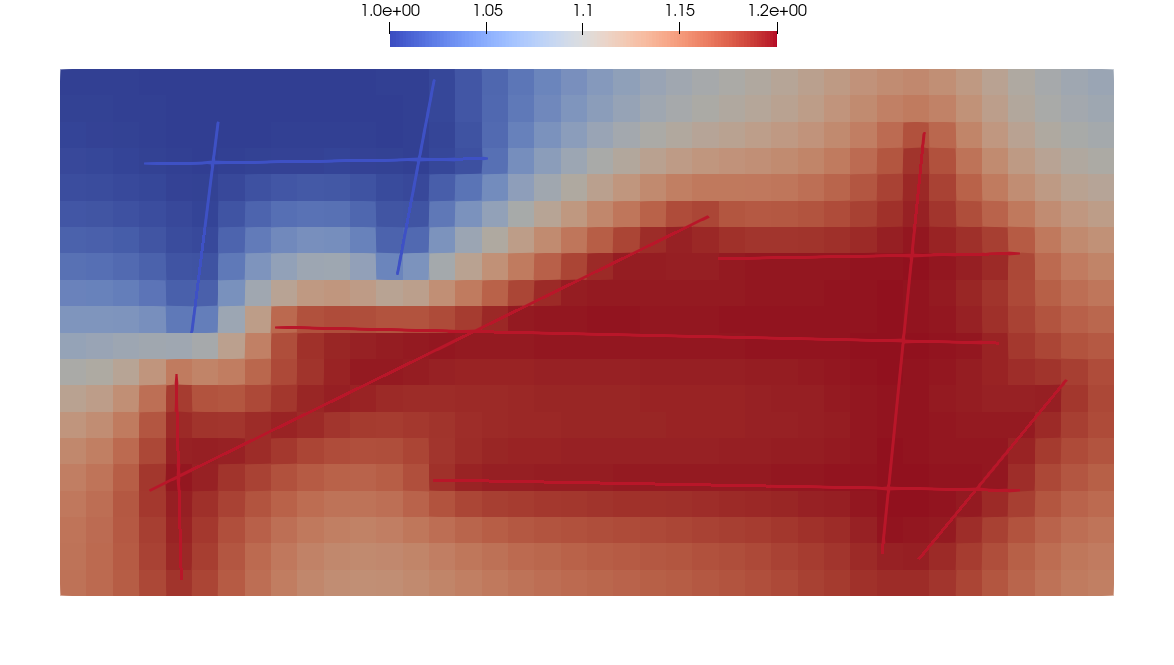}\\
\includegraphics[width=0.19 \textwidth]{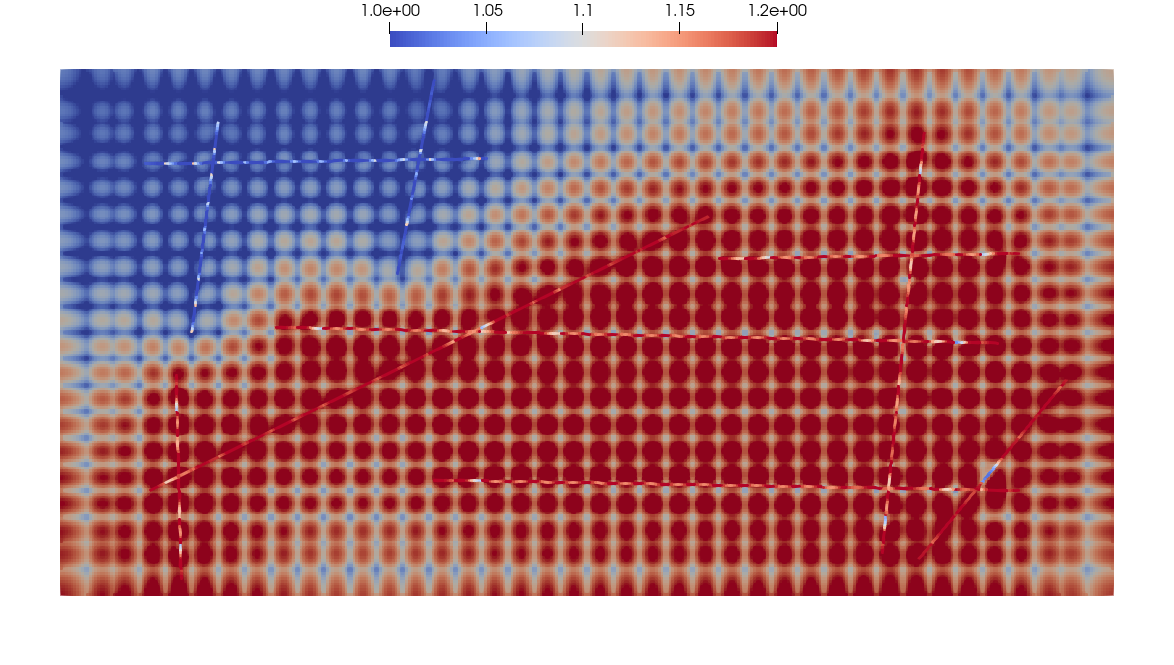}
\includegraphics[width=0.19 \textwidth]{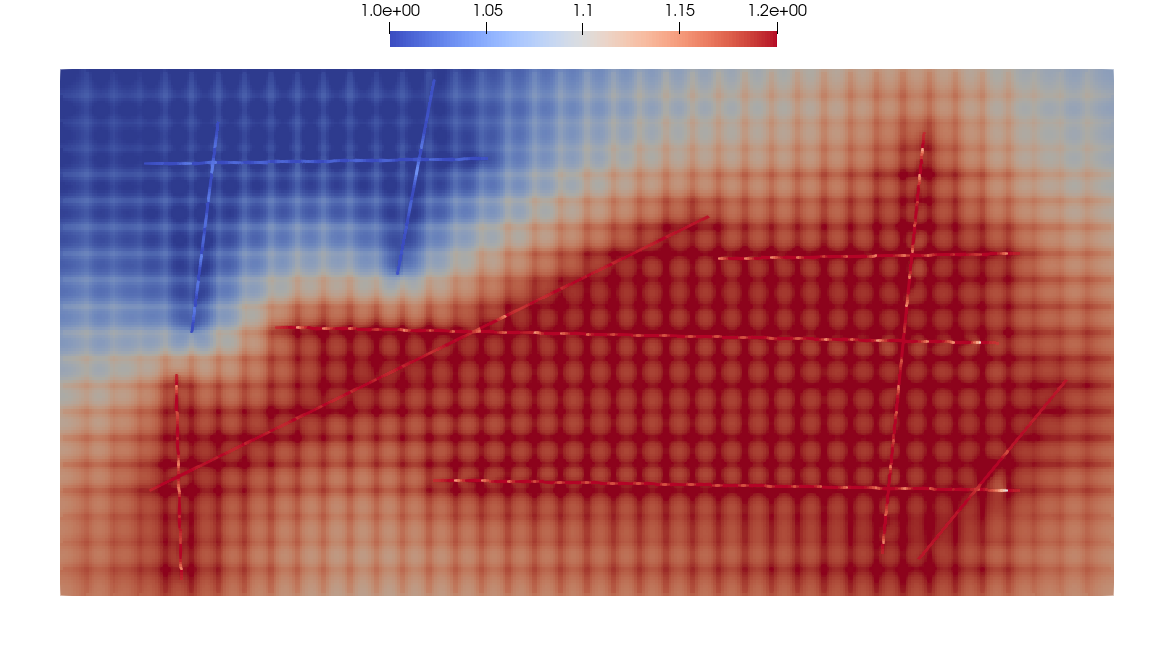}
\includegraphics[width=0.19 \textwidth]{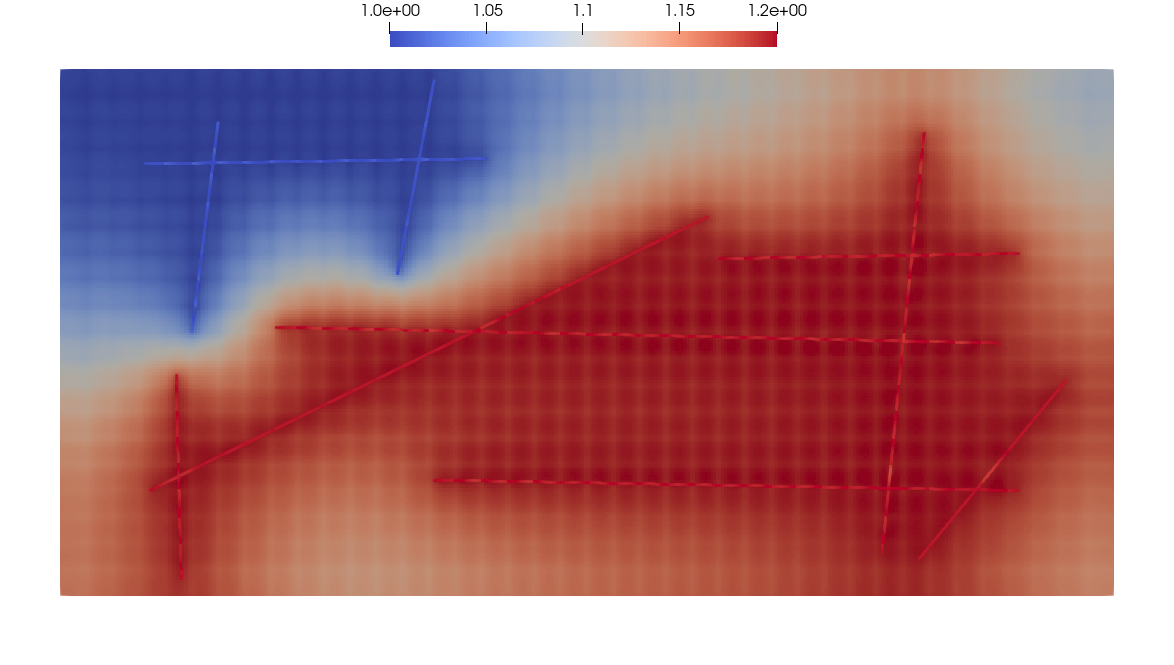}
\includegraphics[width=0.19 \textwidth]{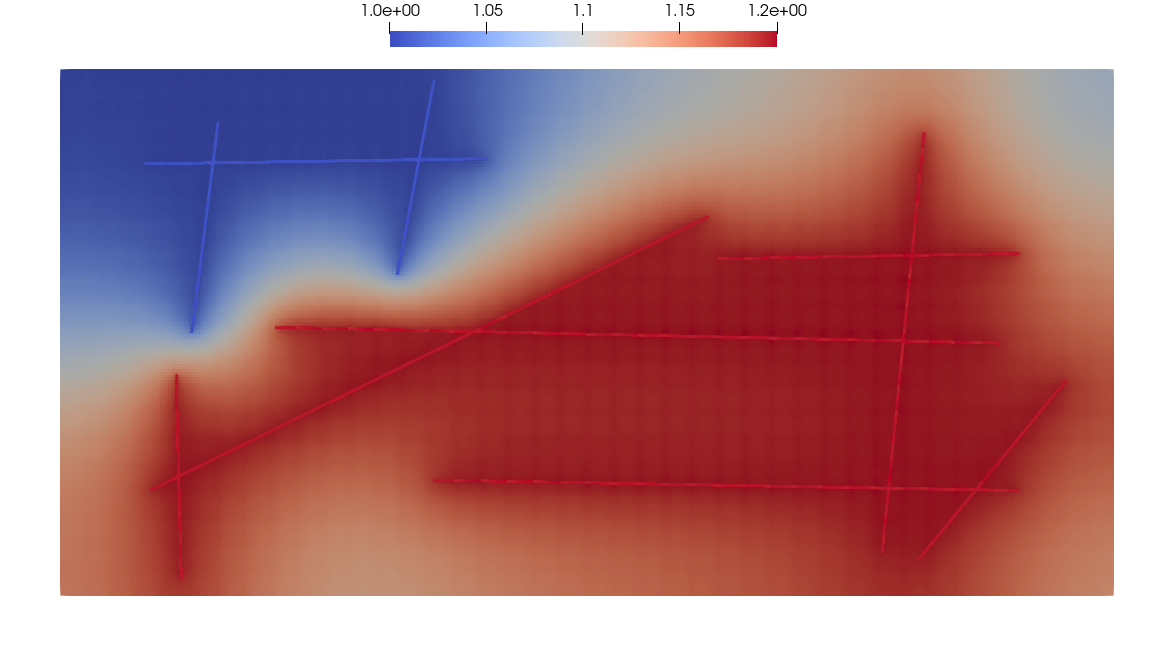}
\includegraphics[width=0.19 \textwidth]{u1024}
\caption{Two-continuum media (\textit{2C}). Multiscale method solution (coarse grid, coupled scheme Im1 with $N_t = 128$) at final time for different number of oversampling layers $K_i^{+,l}$ for $l=2,3,4$ and $5$ (from left to right)}
\label{fig:u2-ms-ov}
\end{figure}

In NLMC approximation, we significantly influence the local domain size used in multiscale basis calculations to the method's accuracy. This study considers a two-continuum test problem to choose an optimal number of oversampling layers in basis calculations. In Table \ref{tab2-ms-ov}, we present the influence of the number of oversampling layers $K_i^{+,l}$ ($l=2,3,4$ and $5$) on the method accuracy for the fine and coarse grid multiscale solutions, $u_{ms}$ and $u_H$ at the final time. 
In Figure \ref{fig:u2-ms-ov},  we present fine grid and coarse grid multiscale solutions at the final time on the first and second rows, respectively. Combining the visual representation of the solution from Figure \ref{fig:u2-ms-ov} and relative errors from Table \ref{tab2-ms-ov}, we can observe how several oversamplings affect the accuracy. Here, we observe that one cannot obtain a solution for the NLMC method by using one oversampling layer in local domain construction. When we use $l=2$ oversampling layers, we have a good approximation of the coarse grid solution. However, two oversampling layers are not enough for good fine-grid resolution of the solution. We should use five oversampling layers to obtain smooth results on the fine grid with a small $e_{ms,1}$ error. However, because we are more interested in the coarse grid model, we can say that the solution is good with less than one \% of error if we use three oversampling layers. We should mention that the number of oversampling layers directly affects the time of multiscale basis calculations. However, the calculations are performed in a parallel way, i.e., for each local domain independently.

\begin{figure}[h!]
\centering
\includegraphics[width=0.32 \textwidth]{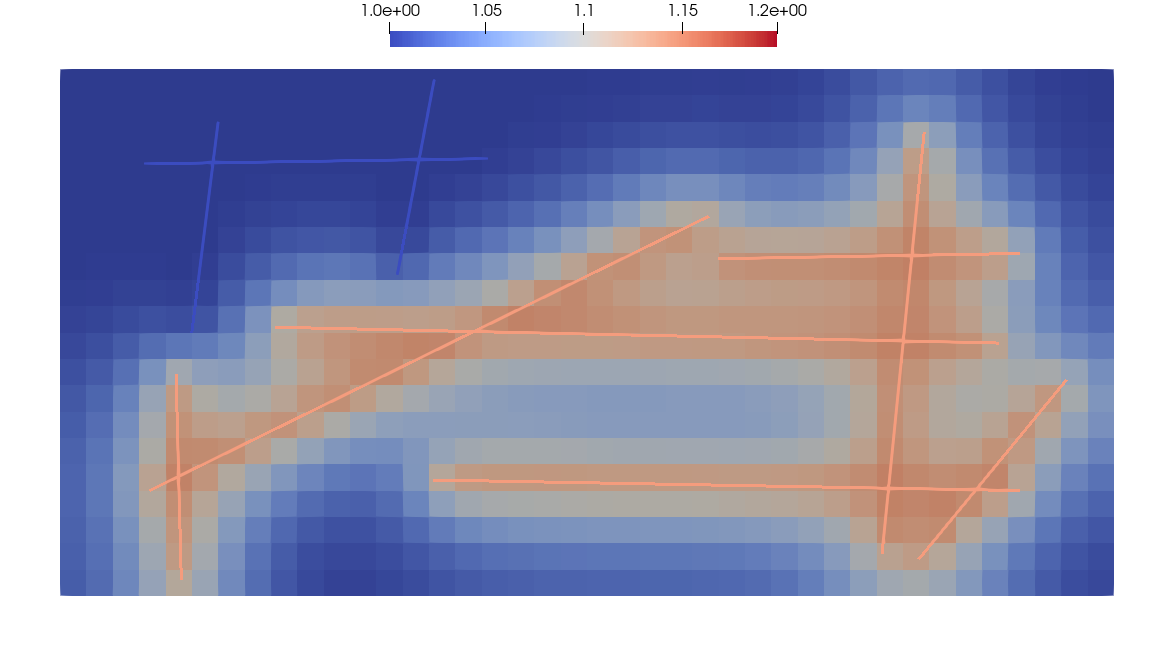}
\includegraphics[width=0.32 \textwidth]{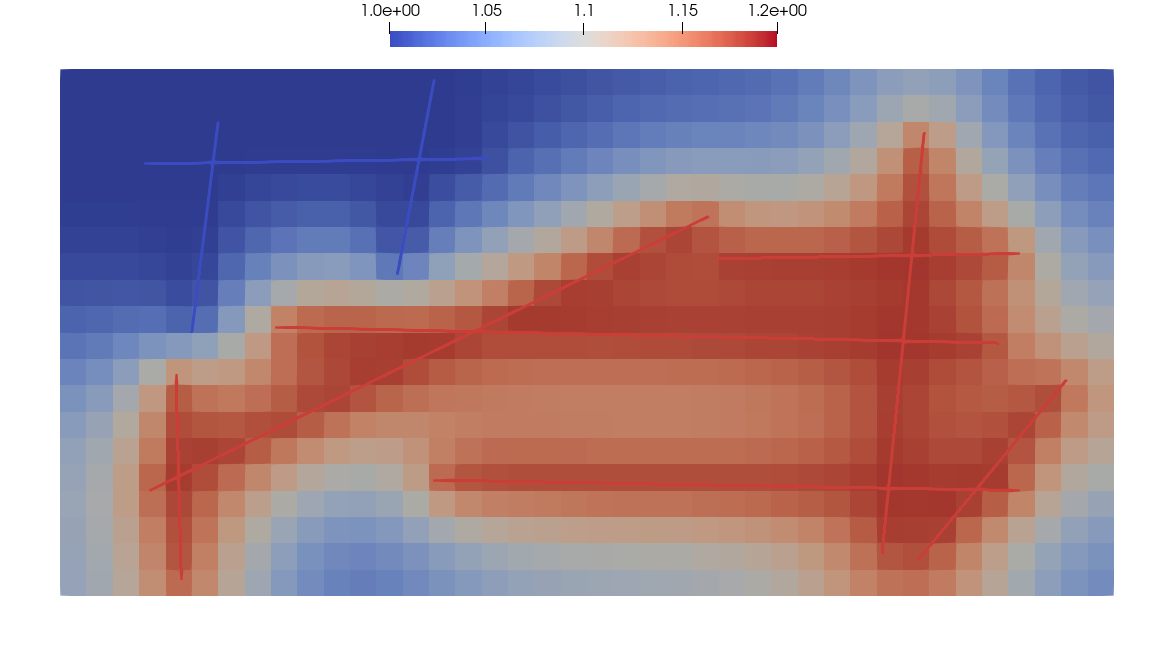}
\includegraphics[width=0.32 \textwidth]{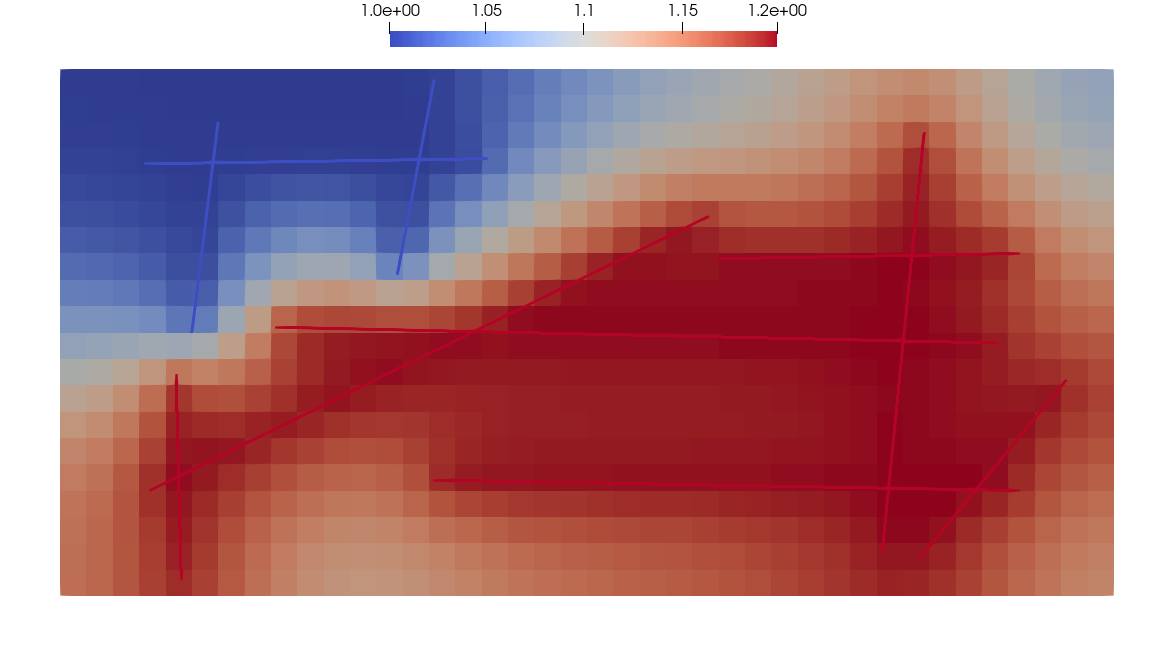}
\caption{Two-continuum media (\textit{2C}). Multiscale method solution (coarse grid, decoupled scheme, ImEx2) for $t = T_{max}/4, T_{max}/2$ and $T_{max}$ (from left to right)}
\label{fig:u2-ms}
\end{figure}

\begin{figure}[h!]
\centering
\includegraphics[width=0.32 \textwidth]{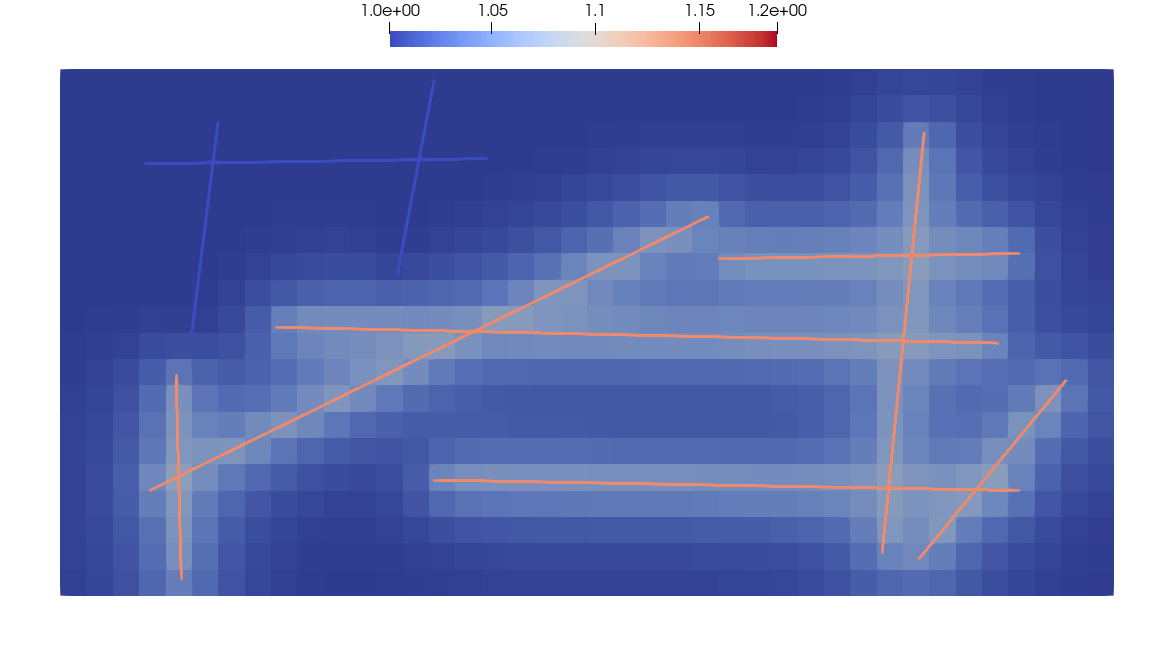}
\includegraphics[width=0.32 \textwidth]{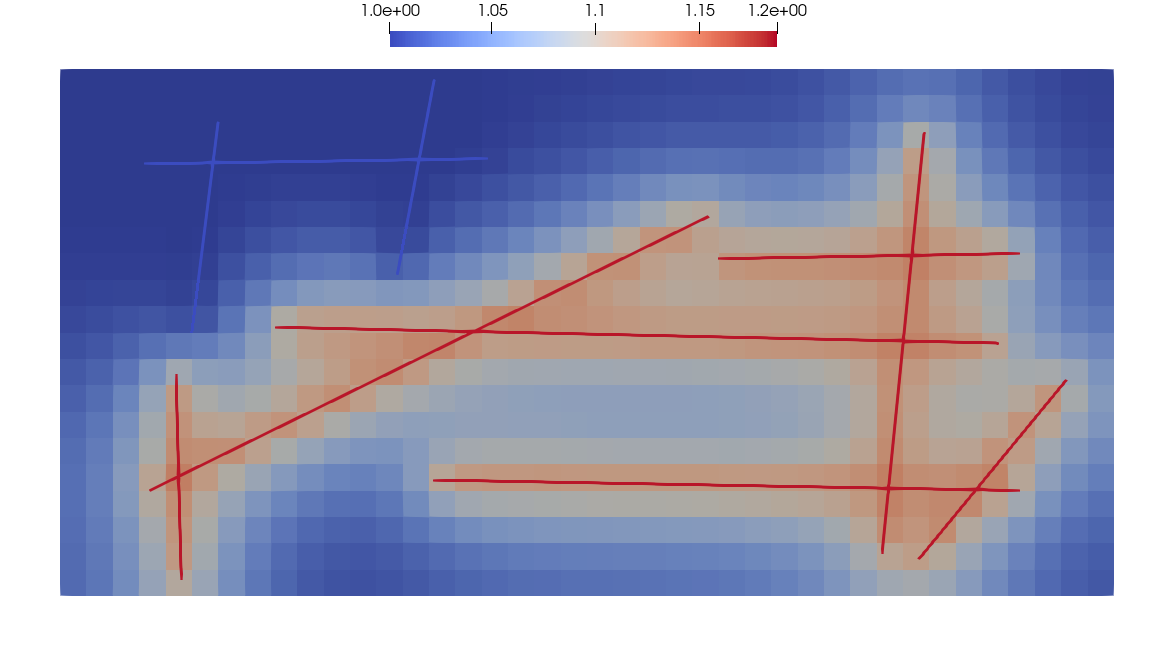}
\includegraphics[width=0.32 \textwidth]{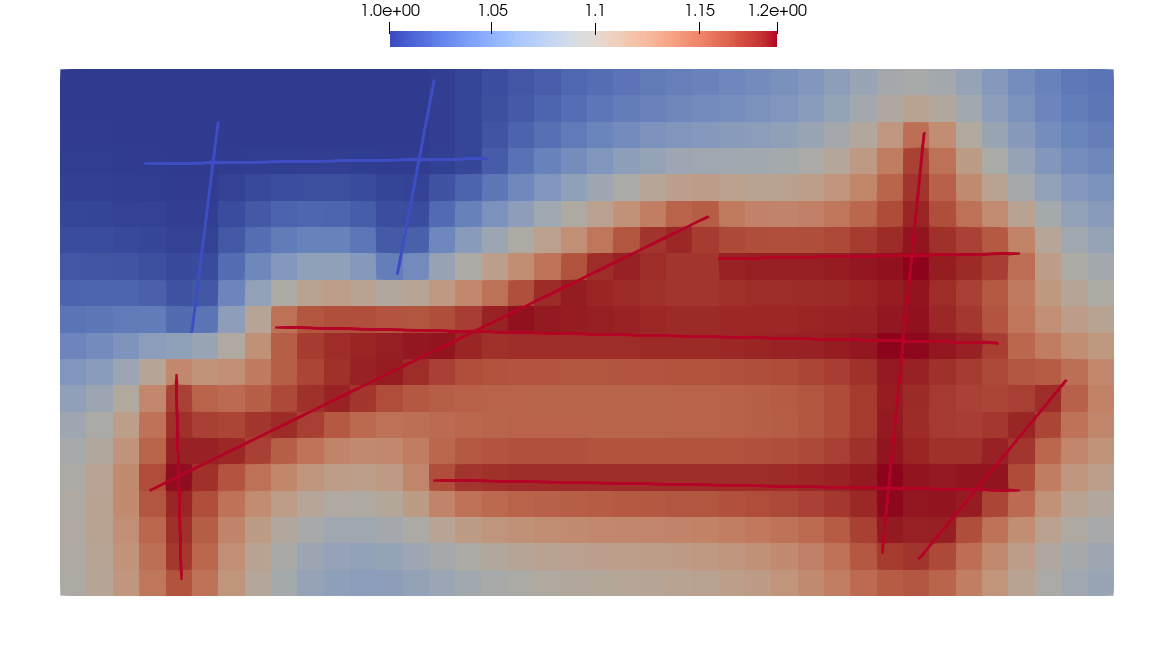}\\
\includegraphics[width=0.32 \textwidth]{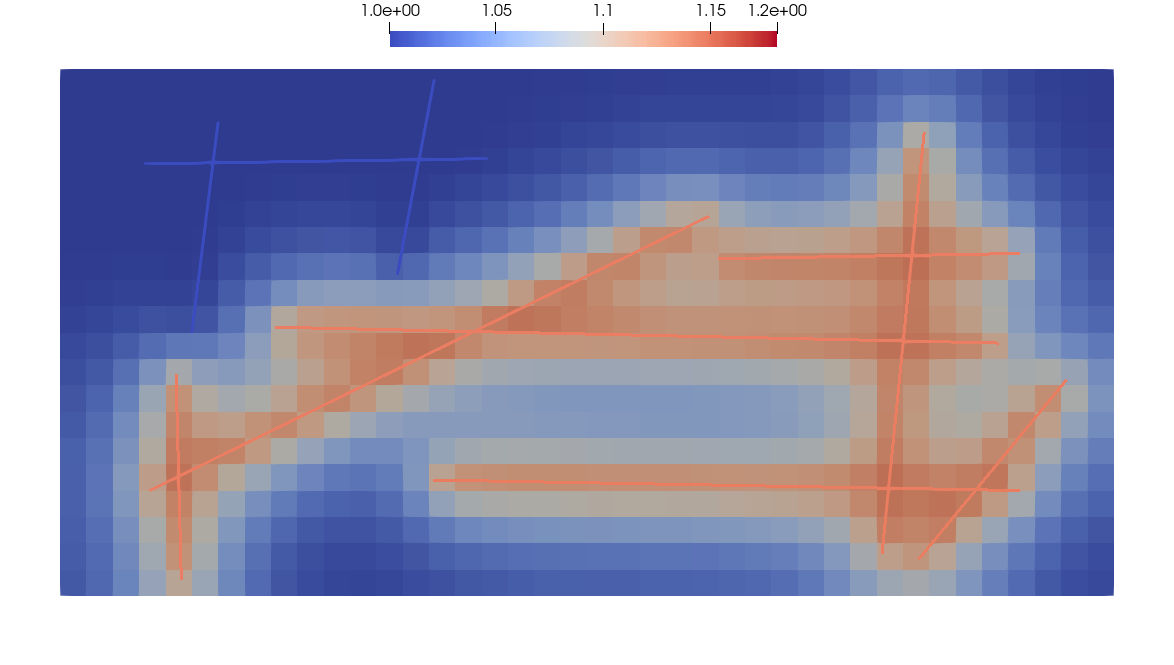}
\includegraphics[width=0.32 \textwidth]{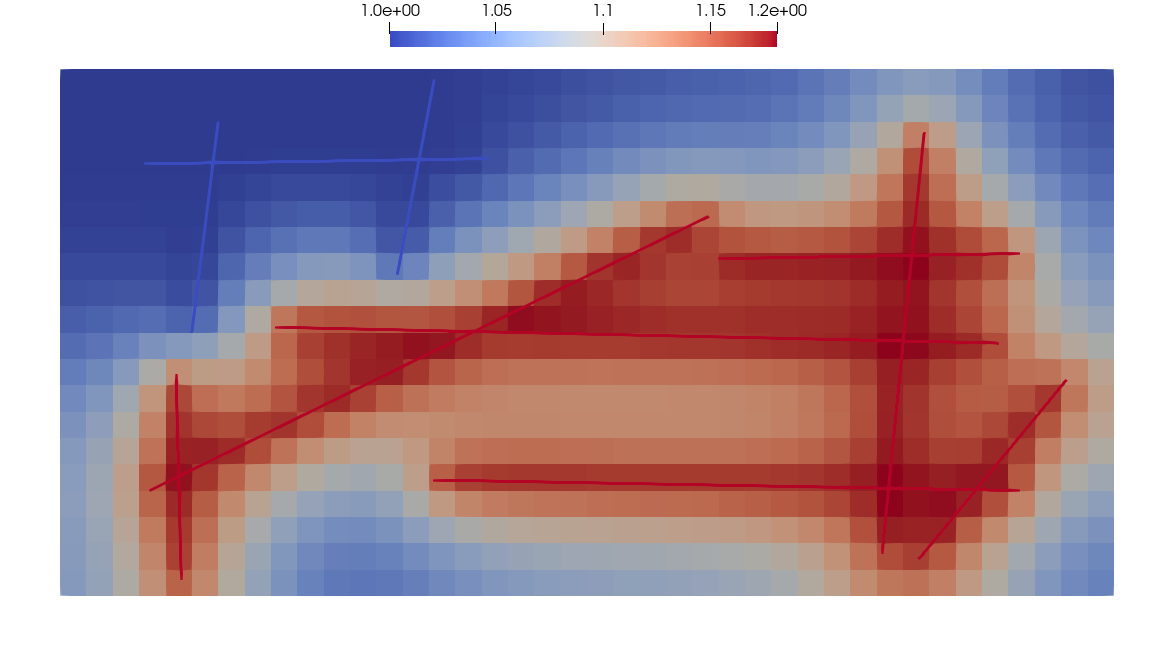}
\includegraphics[width=0.32 \textwidth]{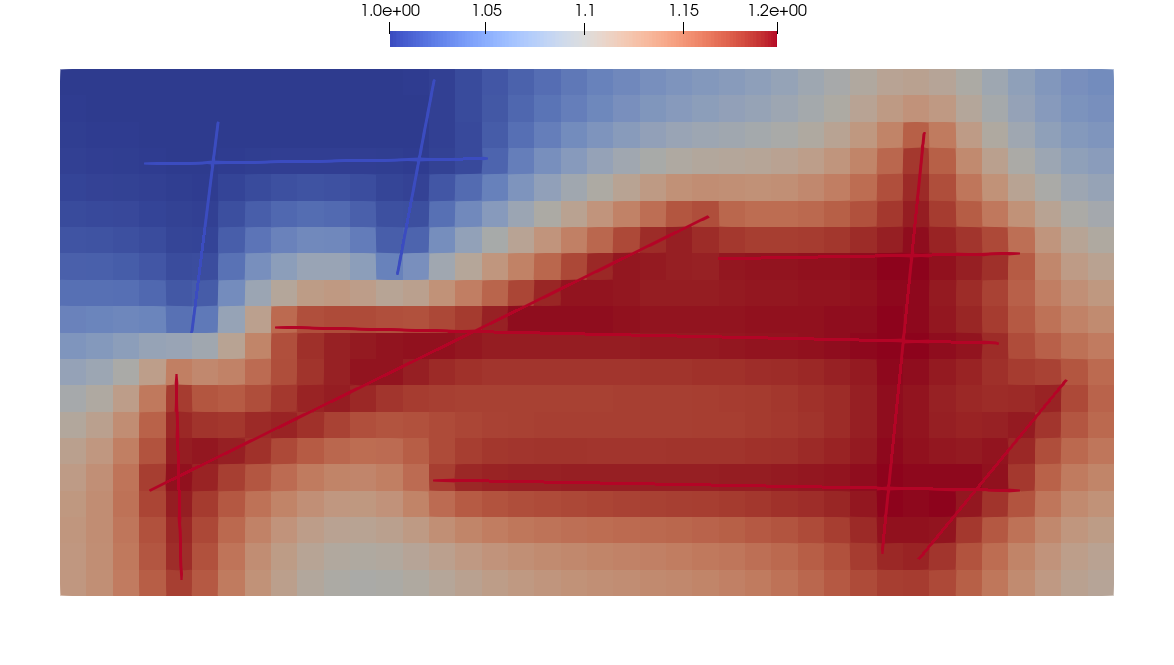}
\caption{Three-continuum media (\textit{3C}). Multiscale method solution (coarse grid, decoupled scheme, ImEx2) for $t = T_{max}/4, T_{max}/2$ and $T_{max}$ (from left to right). First row: fist continuum. Second row: second continuum}
\label{fig:u3-ms}
\end{figure}

Next, in Figures \ref{fig:u2-ms} and \ref{fig:u3-ms}, we depict the solution of the coarse grid system for two and three-continuum test problems, respectively. We use an NLMC approximation with three oversampling layers on a $40\times 20$ coarse grid. The resulting unknowns on the coarse grid are $N_H = 956$ and $N_H = 1756$ for \textit{2C} and \textit{3C} tests in the coupled scheme. The results are shown for the coupled scheme using the two-level  coupled scheme (Ms-Im1) with $N_t=128$ and represented for $n=32, 64$, and $128$.
For the \textit{2C} test in Figure \ref{fig:u2-ms}, we have $e_{H,1} = 0.12, 0.8$ and $0.04$ for $t = T_{max}/4, T_{max}/2$ and $T_{max}$, respectively. 
For the \textit{3C} test in Figure \ref{fig:u3-ms}, we have $e_{H,1} = 0.09, 0.08$ and $0.05$ for $t = T_{max}/4, T_{max}/2$ and $T_{max}$, respectively. 
Solution time is 0.9 sec for \textit{2C} and 2.8 for \textit{3C} using $N_t = 128$ on coarse grid with $N_H = 956$ and $N_H = 1756$ for \textit{2C} and \textit{3C} tests, respectively. 
We have a speedy solution compared to the fine grid system using the same scheme (Im1). On the fine grid, we have 25.3 sec for  \textit{2C} and 53.6 for \textit{3C} for $N_t = 128$.

\begin{table}[h!]
\centering
\begin{tabular}{|c|cc|}
\hline
 \multicolumn{3}{|c|}{Implicit schemes + Ms} \\ 
\hline
\multirow{ 2}{*}{$N_t$}
&  & \\
& $e_{ms,1}$ & $e_{H,1}$ \\
\hline
\multicolumn{3}{|c|}{Ms-Im1} \\
\hline	
4 	
& 1.2637 & 0.9429	\\
8 	
& 0.9254 & 0.4864	\\
16 	
& 0.8098 & 0.2491	\\
32 	
& 0.7769 & 0.1289	\\
64 	
& 0.7681 & 0.0691	\\
128 
& 0.7658 & 0.0401		\\
\hline
\multicolumn{3}{|c|}{Ms-Im2-CN} \\
\hline
4 	
& 1.9531 & 1.9128 \\
8 	
& 1.5967 & 1.4458 \\
16 	
& 1.5556 & 1.3494 \\
32 	
& 1.5548 & 1.3282 \\
64 	
& 1.5586 & 1.3234 \\
128 
& 1.5615 & 1.3223 \\
\hline
\multicolumn{3}{|c|}{Ms-Im2-BDF} \\
\hline
4 	
& 4.1302 & 3.7634	\\
8 
& 1.6554 & 1.3504	\\
16 	
& 0.9912 & 0.5817	\\
32 	
& 0.8223 & 0.2801	\\
64 	
& 0.7799 & 0.1431	\\
128 
& 0.7691 & 0.0779\\
\hline
\end{tabular}
\ \ \ \
\begin{tabular}{|c|cc|cc|cc|}
\hline
 \multicolumn{7}{|c|}{Implicit-Explicit schemes + Ms} \\ 
\hline
\multirow{ 2}{*}{$N_t$}
& \multicolumn{2}{|c|}{L} 
& \multicolumn{2}{|c|}{D}
& \multicolumn{2}{|c|}{U}
\\
& $e_{ms,1}$ & $e_{H,1}$ 
& $e_{ms,1}$ & $e_{H,1}$  
& $e_{ms,1}$ & $e_{H,1}$   \\
\hline
 \multicolumn{7}{|c|}{Ms-ImEx1} \\ 
\hline 
4 
& 2.9032 & 2.5825 
& 3.1320 & 2.8103
& 1.4203 & 1.1282\\
8 
& 1.4942 & 1.1840
& 1.5961 & 1.2948
& 0.9791 & 0.5725\\
16 
& 0.9794 & 0.5651
& 1.0118 & 0.6126
& 0.8250 & 0.2897\\
32 
& 0.8218 & 0.2790
& 0.8304 & 0.3002
& 0.7808 & 0.1485\\
64 
& 0.7798 & 0.1424
& 0.7819 & 0.1522
& 0.7691 & 0.0786\\
128 
& 0.7689 & 0.0761
& 0.7695 & 0.0808
& 0.7661 & 0.0446\\
\hline
 \multicolumn{7}{|c|}{Ms-ImEx2-CNAB} \\ 
\hline 
4 
& 4.7519 & 4.0709
& 9.8183 & 8.4932
& 2.0231 & 1.6289\\
8 
& 1.1839 & 0.7780
& 3.6943 & 3.0293
& 0.8564 & 0.3135\\
16 
& 0.8151 & 0.2601
& 0.8531 & 0.3401
& 0.7695 & 0.0749\\
32 
& 0.7770 & 0.1277
& 0.7806 & 0.1456
& 0.7652 & 0.0283\\
64 
& 0.7680 & 0.0684
& 0.7689 & 0.0770
& 0.7649 & 0.0161\\
128 
& 0.7658 & 0.0409
& 0.7661 & 0.0449
& 0.7648 & 0.0154\\
\hline
 \multicolumn{7}{|c|}{Ms-ImEx2-SBDF} \\ 
\hline 
4 
& 6.0360 & 5.4880
& 6.7632 & 6.2137
& 2.2684 & 1.9512\\
8 
& 1.6974 & 1.3915
& 1.8356 & 1.5321
& 0.9277 & 0.4841\\
16
& 0.9805 & 0.5657
& 1.0122 & 0.6114
& 0.7993 & 0.2168\\
32 
& 0.8179 & 0.2689
& 0.8257 & 0.2886
& 0.7750 & 0.1190\\
64 
& 0.7785 & 0.1364
& 0.7804 & 0.1455
& 0.7680 & 0.0680 \\
128 
& 0.7687 & 0.0743
& 0.7692 & 0.0786
& 0.7659 & 0.0422\\
\hline
\end{tabular}
\caption{Two-continuum media (\textit{2C}). Relative errors in \% for coupled (Implicit) and decoupled (Implicit-Explicit, ImEx) schemes with multiscale approximation}
\label{tab2-ms}
\end{table}

\begin{table}[h!]
\centering
\begin{tabular}{|c|cc|}
\hline
 \multicolumn{3}{|c|}{Implicit schemes + Ms} \\ 
\hline
\multirow{ 2}{*}{$N_t$}
&   &\\
& $e_{ms,1}$ & $e_{H,1}$ \\
\hline
\multicolumn{3}{|c|}{Ms-Im1} \\
\hline	
4 & 1.3455 & 1.0230 \\
8 & 1.0043 & 0.5469 \\
16 & 0.8797 & 0.2858 \\
32 & 0.8426 & 0.1497 \\
64 & 0.8325 & 0.0814 \\
128 & 0.8299 & 0.0490 \\
\hline
\multicolumn{3}{|c|}{Ms-Im2-CN} \\
\hline
4 & 2.7302 & 2.9878 \\
8 & 3.3793 & 3.4420 \\
16 & 3.5427 & 3.5034 \\
32 & 3.5863 & 3.5040 \\
64 & 3.6011 & 3.4994 \\
128 & 3.6070 & 3.4961 \\
\hline
\multicolumn{3}{|c|}{Ms-Im2-BDF} \\
\hline
4 & 4.7068 & 4.4402 \\
8 & 2.0320 & 1.7732 \\
16 & 1.1619 & 0.7808 \\
32 & 0.9174 & 0.3795 \\
64 & 0.8525 & 0.1955 \\
128 & 0.8356 & 0.1077 \\
\hline
\end{tabular}
\ \ \ \
\begin{tabular}{|c|cc|cc|cc|}
\hline
 \multicolumn{7}{|c|}{Implicit-Explicit schemes + Ms} \\ 
\hline
\multirow{ 2}{*}{$N_t$}
& \multicolumn{2}{|c|}{L} 
& \multicolumn{2}{|c|}{D}
& \multicolumn{2}{|c|}{U}\\
& $e_{ms,1}$ & $e_{H,1}$ 
& $e_{ms,1}$ & $e_{H,1}$ 
& $e_{ms,1}$ & $e_{H,1}$   \\
\hline
 \multicolumn{7}{|c|}{Ms-ImEx1} \\ 
\hline 
4 
& 5.0779 & 4.7829
& 5.4451 & 5.1400
& 2.0967 & 1.8501\\
8 
& 2.6295 & 2.3841
& 2.9927 & 2.7472
& 1.3463 & 1.0186\\
16 
& 1.4840 & 1.1777
& 1.6917 & 1.4104
& 1.0015 & 0.5414\\
32 
& 1.0311 & 0.5886
& 1.1169 & 0.7179
& 0.8790 & 0.2845\\
64 
& 0.8849 & 0.2997
& 0.9123 & 0.3678
& 0.8428 & 0.1515\\
128 
& 0.8441 & 0.1575
& 0.8518 & 0.1924
& 0.8328 & 0.0848 \\
\hline
 \multicolumn{7}{|c|}{Ms-ImEx2-CNAB} \\ 
\hline 
4 
& 7.9192 & 7.3925
& 10.8835 & 10.0329
& 3.4384 & 3.1184\\
8 
& 1.0533 & 0.5766
& 3.4144 & 2.8888
& 1.5138 & 1.1977\\
16 
& 0.8407 & 0.1345
& 0.9445 & 0.4295
& 1.0113 & 0.5502\\
32 
& 0.8315 & 0.0658
& 0.8537 & 0.1997
& 0.8728 & 0.2596\\
64 
& 0.8295 & 0.0405
& 0.8355 & 0.1067
& 0.8386 & 0.1213\\
128 
& 0.8291 & 0.0313
& 0.8308 & 0.0634
& 0.8307 & 0.0555\\
\hline
 \multicolumn{7}{|c|}{Ms-ImEx2-SBDF} \\ 
\hline 
4 
& 10.7971 & 10.2690
& 10.4765 & 9.9767
& 2.3070 & 2.0212\\
8 
& 1.9309 & 1.6661
& 2.5572 & 2.3095
& 1.1612 & 0.7747\\
16 
& 1.0026 & 0.5421
& 1.2120 & 0.8478
& 0.8713 & 0.2571\\
32 
& 0.8676 & 0.2488
& 0.9243 & 0.3948
& 0.8356 & 0.1038\\
64 
& 0.8388 & 0.1278
& 0.8535 & 0.1995
& 0.8298 & 0.0449\\
128 
& 0.8318 & 0.0735
& 0.8358 & 0.1088
& 0.8288 & 0.0251\\
\hline
\end{tabular}
\caption{Three-continuum media (\textit{3C}). Relative errors in \% for coupled (Implicit, Im) and decoupled (Implicit-Explicit, ImEx) schemes with multiscale approximation}
\label{tab3-ms}
\end{table}

Tables \ref{tab2-ms} and \ref{tab3-ms} show multiscale method errors at final time for two and three-continuum test problems, respectively.
We observe the convergence behavior of the proposed Implicit and Implicit-Explicit schemes for multiscale approximation from the result. 
We observe excellent performance on the coarse grid approximation using the NLMC method for the first-order coupled and decoupled schemes. We can use fewer time steps and a coarse grid formulation to produce an accurate solution. For example, we have near 0.5 \% of $e_{H,1}$ error using coupled Ms-Im1 scheme and decoupled version Ms-imEx1 using only eight time steps for \textit{2C} test (see Table \ref{tab2-ms}).
However, for both two and three-continuum tests, we observe a lousy performance of the second-order coupled scheme based on the Crank-Nicolson approximation (Ms-Im2-CN), where we have fundamental error $e_{H,1} = 1.3$ \% for \textit{2C} and  $e_{H,1} = 3.5$ \% for \textit{3C} for $N_t=128$. However, we can obtain good results using decoupled implicit-explicit variation (Ms-Im2-CNAB-D) with  $e_{H,1} = 0.04$ \% for \textit{2C} and  $e_{H,1} = 0.06$ \% for \textit{3C} for $N_t=128$. Furthermore, we can obtain a perfect solution using only eight-time steps with less than one percent of $e_{H,1}$ error using the Ms-Im2-CNAB-L scheme.
For the coupled Ms-Im2-BDF scheme and decoupled Ms-ImEx2-SBDF schemes, we obtain a little bit better results in decoupled case with near 0.5 \% of $e_{H,1}$ error using $N_t=8$ in   Ms-ImEx2-SBDF and using $N_t=16$ in Ms-Im2-BDF in \textit{2C} test. In \textit{3C} test, we observe a similar behavior with  near 0.8 \% of $e_{H,1}$ error using $N_t=8$ in   Ms-ImEx2-SBDF and using $N_t=16$ in Ms-Im2-BDF. 
By comparing the L, D, and U versions of schemes, we see that the U-scheme works better for  Ms-ImEx1 and Ms-ImEx2-SBDF schemes.

\begin{table}[h!]
\centering
\begin{tabular}{|c|cc|}
\hline
 \multicolumn{3}{|c|}{Implicit + Ms} \\ 
\hline
$N_t$ 
& time$_{tot}$ & $N_{AvIt}$ \\
\hline
\multicolumn{3}{|c|}{Ms-Im1} \\
\hline
4 	& 0.057 & 15.0 \\
8 	& 0.094 & 14.0 \\
16 	& 0.160 & 13.9 \\
32 	& 0.252 & 14.0 \\
64 	& 0.460 & 13.8 \\
128 & 0.902 & 13.4 \\
\hline
\multicolumn{3}{|c|}{Ms-Im2-CN} \\
\hline
4 	& 0.047 & 14.0  \\
8 	& 0.070 & 13.7  \\
16 	& 0.119 & 14.0  \\
32 	& 0.214 & 13.6  \\
64 	& 0.403 & 13.3  \\
128 & 0.752 & 13.0  \\
\hline
\multicolumn{3}{|c|}{Ms-Im2-BDF} \\
\hline
4 	& 0.047 & 15.0 \\
8 	& 0.071 & 14.0 \\
16 	& 0.127 & 13.9 \\
32 	& 0.215 & 13.8 \\
64 	& 0.402 & 13.6 \\
128 & 0.754 & 13.0 \\
\hline
\end{tabular}
\ \ \ \
\begin{tabular}{|c|c|cc|cc|}
\hline
 \multicolumn{6}{|c|}{Implicit-Explicit + Ms} \\ 
\hline
$N_t$ 
& time$_{tot}$ 
& time$_{tot}^m$ & $N_{AvIt}^m$ 
& time$_{tot}^f$ & $N_{AvIt}^f$   \\
\hline
 \multicolumn{6}{|c|}{Ms-ImEx1} \\ 
\hline 
4 	& 0.004 & 0.003 & 3.0 & 0.0007 & 14.0 \\
8 	& 0.009 & 0.007 & 3.0 & 0.0016 & 14.0 \\
16 	& 0.019 & 0.016 & 3.0 & 0.0032 & 14.0 \\
32 	& 0.030 & 0.023 & 2.0 & 0.0066 & 14.0 \\
64 	& 0.058 & 0.045 & 2.0 & 0.0129 & 13.9 \\
128 & 0.120 & 0.095 & 2.0 & 0.0259 & 13.6 \\
\hline
 \multicolumn{6}{|c|}{Ms-ImEx2-CNAB} \\
\hline 
4 	& 0.004 & 0.003 & 3.0 & 0.0007 & 14.0 \\
8 	& 0.009 & 0.008 & 3.0 & 0.0016 & 14.0 \\
16 	& 0.017 & 0.013 & 2.0 & 0.0037 & 14.0 \\
32 	& 0.036 & 0.028 & 2.0 & 0.0075 & 13.9 \\
64 	& 0.072 & 0.057 & 2.0 & 0.0151 & 13.8 \\
128 & 0.151 & 0.122 & 2.0 & 0.0289 & 12.4 \\
\hline
 \multicolumn{6}{|c|}{Ms-ImEx2-SBDF} \\
\hline 
4 	& 0.004 & 0.003 & 3.0 & 0.0007 & 14.0 \\
8 	& 0.009 & 0.008 & 3.0 & 0.0017 & 14.0 \\
16 	& 0.017 & 0.013 & 2.0 & 0.0036 & 14.0 \\
32 	& 0.037 & 0.029 & 2.0 & 0.0075 & 14.0 \\
64 	& 0.074 & 0.058 & 2.0 & 0.0155 & 13.8 \\
128 & 0.146 & 0.116 & 2.0 & 0.0299 & 13.2 \\
\hline
\end{tabular}
\caption{Two-continuum media (\textit{2C}). Time of the solution and the average number of iterations.  Coupled (Implicit, Im) and decoupled (Implicit-Explicit, ImEx)}
\label{tab2-ms-t}
\end{table}

\begin{table}[h!]
\centering
\begin{tabular}{|c|cc|}
\hline
 \multicolumn{3}{|c|}{Implicit + Ms} \\ 
\hline
$N_t$ 
& time$_{tot}$ & $N_{AvIt}$ \\
\hline
\multicolumn{3}{|c|}{Ms-Im1} \\
\hline
4 	& 0.391 & 15.0  \\
8 	& 0.340 & 14.0  \\
16 	& 0.531 & 14.0  \\
32 	& 0.839 & 13.9  \\
64 	& 1.513 & 13.6  \\
128 & 2.800 & 13.1  \\
\hline
\multicolumn{3}{|c|}{Ms-Im2-CN} \\
\hline
4 	& 0.240 & 14.0  \\
8 	& 0.321 & 14.0  \\
16 	& 0.483 & 14.0  \\
32 	& 0.789 & 13.6  \\
64 	& 1.427 & 13.2  \\
128 & 2.438 & 12.0  \\
\hline
\multicolumn{3}{|c|}{Ms-Im2-BDF} \\
\hline
4 	& 0.241 & 15.0  \\
8 	& 0.321 & 14.0  \\
16 	& 0.485 & 14.0  \\
32 	& 0.820 & 13.8  \\
64 	& 1.416 & 13.4  \\
128 & 2.596 & 13.0  \\
\hline
\end{tabular}
\ \ \ \
\begin{tabular}{|c|c|cc|cc|cc|}
\hline
 \multicolumn{8}{|c|}{Implicit-Explicit + Ms} \\ 
\hline
$N_t$ 
& time$_{tot}$ 
& time$_{tot}^1$ & $N_{AvIt}^1$ 
& time$_{tot}^2$ & $N_{AvIt}^2$ 
& time$_{tot}^f$ & $N_{AvIt}^f$   \\
\hline
 \multicolumn{8}{|c|}{Ms-ImEx1} \\ 
\hline 
4 	& 0.006 & 0.002 & 1.0 & 0.004 & 3.0 & 0.0008 & 14.0 \\
8 	& 0.013 & 0.004 & 1.0 & 0.008 & 3.0 & 0.0016 & 14.0 \\
16 	& 0.026 & 0.008 & 1.0 & 0.014 & 2.6 & 0.0032 & 14.0 \\
32 	& 0.047 & 0.016 & 1.0 & 0.024 & 2.0 & 0.0065 & 14.0 \\
64 	& 0.094 & 0.032 & 1.0 & 0.048 & 2.0 & 0.0129 & 13.9 \\
128 & 0.188 & 0.064 & 1.0 & 0.097 & 2.0 & 0.0258 & 13.7 \\
\hline
 \multicolumn{8}{|c|}{Ms-ImEx2-CNAB} \\
\hline 
4 	& 0.006 & 0.002 & 1.0 & 0.003 & 3.0 & 0.0007 & 14.0 \\
8 	& 0.015 & 0.005 & 1.0 & 0.007 & 2.6 & 0.0017 & 14.0 \\
16 	& 0.030 & 0.011 & 1.0 & 0.015 & 2.0 & 0.0036 & 14.0 \\
32 	& 0.047 & 0.016 & 1.0 & 0.024 & 2.0 & 0.0065 & 13.9 \\
64	& 0.094 & 0.032 & 1.0 & 0.048 & 2.0 & 0.0129 & 13.7 \\
128 & 0.188 & 0.064 & 1.0 & 0.097 & 2.0 & 0.0258 & 12.3 \\
\hline
 \multicolumn{8}{|c|}{Ms-ImEx2-SBDF} \\
\hline 
4	& 0.006 & 0.002 & 1.0 & 0.003 & 3.0 & 0.0007 & 14.0 \\
8 	& 0.016 & 0.005 & 1.0 & 0.008 & 3.0 & 0.0017 & 14.0 \\
16 	& 0.029 & 0.011 & 1.0 & 0.014 & 2.0 & 0.0036 & 14.0 \\
32 	& 0.060 & 0.022 & 1.0 & 0.030 & 2.0 & 0.0075 & 14.0 \\
64 	& 0.138 & 0.051 & 1.0 & 0.069 & 2.0 & 0.0172 & 13.8 \\
128 & 0.267 & 0.101 & 1.0 & 0.133 & 2.0 & 0.0323 & 13.3 \\
\hline
\end{tabular}
\caption{Three-continuum media (\textit{3C}). Time of the solution and the average number of iterations. Coupled (Implicit, Im) and decoupled (Implicit-Explicit, ImEx)}
\label{tab3-ms-t}
\end{table}

Finally, we discuss the computational performance of the presented combination of the time and space approximation schemes, where we have a substantial computational time reduction based on the coarse grid calculations using the NLMC method. Moreover, we can decouple a system and use a more significant time step size using presented Implicit-Explicit time approximation schemes. 
Similarly to the fine grid calculations, we present the total time of computations in seconds with an average number of iterations for the preconditioned conjugate gradient method in Tables \ref{tab2-ms-t} and \ref{tab3-ms-t}.  For the decoupled schemes, the results are presented for U-scheme.
For the coupled first-order scheme (Im1), we have the following:
\begin{itemize}
\item 
Two-continuum media (\textit{2C}): 
\begin{itemize}
\item Reference solution with $N_h = 81474$ (fine grid) with $N_t = 1024$ (Im1): time$_{tot} = 103.76$ sec. and $N_{AverIt} = 73.98$. 
\item Fine-scale solution with $N_h = 81474$ (fine grid) with $N_t = 128$ (Im1): time$_{tot} = 25.31$ sec. and $N_{AverIt} = 146.5$. 
\item Multiscale solution with $N_h = 956$ (coarse grid) with $N_t = 128$ (Im1): 
time$_{tot} = 0.90$ sec. and $N_{AverIt} = 13.4$. 
\end{itemize}
\item 
Three-continuum media (\textit{3C}): 
\begin{itemize}
\item Reference solution with $N_h = 161474$ (fine grid) with $N_t = 1024$ (Im1): 
time$_{tot} = 234.37$ sec. and $N_{AverIt} = 73.90$
\item Fine-scale solution with $N_h = 161474$ (fine grid) with $N_t = 128$ (Im1): time$_{tot} = 53.68$ sec. and $N_{AverIt} = 140.1$. 
\item Multiscale solution with $N_h = 1756$ (coarse grid) with $N_t = 128$ (Im1): 
time$_{tot} = 2.8$ sec. and $N_{AverIt} = 13.1$. 
\end{itemize}
\end{itemize}
By constructing the accurate coarse-scale approximation, we reduce the size of the system significantly with 28.1 and 19.2 times faster calculations compared with a fine-scale solution with the same time step size ($N_t = 128$).

Moreover, from the Tables \ref{tab2-ms-t} and \ref{tab3-ms-t}, we observe a significant influence of the decoupling to the time of computations. 
For $N_t = 128$, we have  for the two-continuum test (\textit{2C}): 
\begin{itemize}
\item Fine-scale solution with $N_h = 81474$ (fine grid): 
\begin{itemize}
\item Coupled scheme (Im2-BDF): time$_{tot} = 20.02$ sec. with $N_{AverIt} = 129.7$. 
\item Decoupled scheme (ImEx2-SBDF): time$_{tot} = 7.52$ sec. with $N_{AverIt}^m = 29.0$ and $N_{AverIt}^f = 53.0$. 
\end{itemize}
\item Multiscale solution with $N_h = 956$ (coarse grid):
\begin{itemize}
\item Coupled scheme (Ms-Im2-BDF): time$_{tot} = 0.75$ sec. with $N_{AverIt} = 13.0$. 
\item Decoupled scheme (Ms-ImEx2-SBDF): time$_{tot} = 0.15$ sec. with $N_{AverIt}^m = 2.0$ and $N_{AverIt}^f = 13.2$.
\end{itemize}
\end{itemize}
In a second-order scheme with $(\mu, \sigma) = (1.5, 0)$, we obtain a 2.7 times faster solution on the fine grid for the decoupled system and five times faster solution on the coarse grid with $0.08$ \% of $e_{H,1}$ error compared with a reference solution.
We also observe a significant reduction in the number of iterations for the continuum in $\Omega$, significantly influencing the solution time. 

For the three-continuum media (\textit{3C}): 
\begin{itemize}
\item Fine-scale solution with $N_h = 161474$ (fine grid):
\begin{itemize}
\item Coupled scheme (Im2-BDF): time$_{tot} = 49.03$ sec. with $N_{AverIt} = 125.1$. 
\item Decoupled scheme (ImEx2-SBDF): time$_{tot} = 10.39$ sec. with $N_{AverIt}^1 = 3.0$, $N_{AverIt}^1 = 29.0$ and $N_{AverIt}^f = 52.0$. 
\end{itemize}
\item Multiscale solution with $N_h = 1756$ (coarse grid):
\begin{itemize}
\item Coupled scheme (Ms-Im2-BDF): time$_{tot} = 2.60$ sec. with $N_{AverIt} = 13.0$. 
\item Decoupled scheme (Ms-ImEx2-SBDF): time$_{tot} = 0.27$ sec. with $N_{AverIt}^1 = 1.0$, $N_{AverIt}^1 = 2.0$ and $N_{AverIt}^f = 13.3$. 
\end{itemize}
\end{itemize}
In the second-order scheme for \textit{3C}, we obtain a 4.7 times faster solution on the fine grid for the decoupled system and a 9.6 times faster solution on the coarse grid with $0.1$ \% of $e_{H,1}$ error compared with a reference solution. 
We observe a substantial computational time reduction by combining Implicit-Explicit schemes and accurate multiscale approximation by the NLMC method. We can solve a system in less than a second compared with 103.76 sec for two-continuum and  234.37 sec for three-continuum problems.

\section{Conclusion}

We presented efficient decoupled schemes for multicontinuum problems in porous media. The decoupled schemes are constructed in a general way using Implicit-explicit time approximation schemes. An additive representation of the operator is used to decouple equations for each continuum. We used a finite-volume method for fine-scale approximation, and the nonlocal multicontinuum (NLMC) method was used to construct an accurate and physically meaningful coarse-scale approximation. The NLMC method is based on defining the macroscale variables on the coarse grid for the general multicontinua problem in porous media. We investigated the stability of the two and three-level time approximation schemes. 
We observed that combining decoupling techniques with multiscale approximation leads to developing an efficient solver for multicontinuum problems. An extensive numerical investigation was given for two and three continuum problems that describe flow processes in fractured porous media. By decoupled calculations, we obtain 3-4 times faster calculations for two-continuum media and 5-8 times faster calculations for three-continuum media than a regular implicit approximation leading to the coupled system on the fine grid. Furthermore, by combining two techniques (nonlocal multicontimuum method and continuum decoupling), the simulation time becomes 140-210 times faster for two-continuum and 180-280 times faster for three continuum media compared with a fine-scale coupled schemes of the same order of time approximation with similar time step size.

\bibliographystyle{plain}
\bibliography{lit}

\end{document}